\documentclass{amsart}
\usepackage{epsfig}
\usepackage{amsmath}
\usepackage{amssymb}
\usepackage{amscd}
\usepackage{graphicx}
\usepackage{color}
\usepackage{float}
\usepackage{verbatim}
\usepackage{bbm}
\usepackage{enumerate}
\usepackage{hyperref}
\numberwithin{equation}{section}

\newtheorem{thm}{Theorem}[section]
\newtheorem{lem}[thm]{Lemma}
\newtheorem{prop}[thm]{Proposition}
\newtheorem{fact}[thm]{Fact}

\newtheorem{cor}[thm]{Corollary}
\newtheorem{defn}[thm]{Definition}

\theoremstyle{remark}
\newtheorem{rem}[thm]{Remark}

\newtheorem{exam}[thm]{Example}

\def \N {\mathbb N}
\def \T {\mathcal T}
\def \TT {\mathsf T}
\def \TTT {\boldsymbol{\mathsf T}}
\def \bT {\boldsymbol\T}
\def \O {\mathcal O}
\def \Z {\mathbb Z}
\def \R {\mathcal R}
\def \K {\mathcal K}
\def \F {\mathcal F}
\def \Q {\mathcal Q}

\def \P {\mathcal P}
\def \S {\mathcal S}

\def \eps {\varepsilon}

\def \sq {sequence}

\def \tl {topological}
\def \im {invariant measure}
\def \inv {invariant}
\def \ds {dynamical system}
\def \incr {\mathsf{incr}}

\begin{document}
\title{Multiorders in amenable group actions}

\author[T.\ Downarowicz, P.\ Oprocha, M.\ Wi\c{e}cek and G.\ Zhang]{Tomasz Downarowicz, Piotr Oprocha, Mateusz Wi\c{e}cek and Guohua Zhang}

\address{\vskip 2pt \hskip -12pt Tomasz Downarowicz}

\address{\hskip -12pt Faculty of Pure and Applied Mathematics, Wroc\l aw University of Technology, Wroc\l aw, Poland}

\email{downar@pwr.edu.pl}

\address{\vskip 2pt \hskip -12pt Piotr Oprocha}

\address{\hskip -12pt Faculty of Applied Mathematics, AGH University of Science and Technology, Krak\'ow, Poland}

\email{oprocha@agh.edu.pl}

\address{\vskip 2pt \hskip -12pt Mateusz Wi\c{e}cek}

\address{\hskip -12pt Faculty of Pure and Applied Mathematics, Wroc\l aw University of Technology, Wroc\l aw, Poland}

\email{mateusz.wiecek@pwr.edu.pl}

\address{\vskip 2pt \hskip -12pt Guohua Zhang}

\address{\hskip -12pt School of Mathematical Sciences and Shanghai Center for Mathematical Sciences, Fudan University, Shanghai 200433, China}

\email{chiaths.zhang@gmail.com}

\subjclass[2010]{Primary 37A15, 37A35; Secondary 43A07}
\keywords{Countable amenable group, measure-preserving action, invariant random order, multiorder, conditional entropy, Pinsker factor, orbit equivalence.}

\begin{abstract}  The paper offers a thorough study of multiorders and their applications to measure-preserving actions of countable amenable groups.
	By a~{\em multiorder} on a~countable group we mean any probability measure $\nu$ on the collection $\tilde\O$ of linear orders of type $\Z$ on $G$, invariant under the natural action of $G$ on such orders. Multiorders exist on any countable amenable group (and only on such groups) and every multiorder has the F\o lner property, meaning that almost surely the order intervals starting at the unit form a F\o lner \sq. Every free measure-preserving $G$-action $(X,\mu,G)$ has a~multiorder $(\tilde\O,\nu,G)$ as a factor and has the same orbits as the $\Z$-action $(X,\mu,S)$, where $S$ is the \emph{successor map} determined by the multiorder factor. Moreover, the sub-sigma-algebra $\Sigma_{\tilde\O}$ associated with the multiorder factor is invariant under $S$, which makes the corresponding $\Z$-action $(\tilde\O,\nu,\tilde S)$ a factor of $(X,\mu,S)$. We prove that the entropy of any $G$-process generated by a finite partition of $X$, conditional with respect to $\Sigma_{\tilde\O}$, is preserved by the orbit equivalence with $(X,\mu,S)$. Furthermore, this entropy can be computed in terms of the so-called random past, by a formula analogous to $	h(\mu,T,\P)=H(\mu,\P|\P^-)$ known for $\Z$-actions. The above fact is then applied to prove a variant of a result by Rudolph and Weiss \cite{RW}. The original theorem states that orbit equivalence between free actions of countable amenable groups preserves conditional entropy with respect to a~sub-sigma-algebra $\Sigma$, as soon as the ``orbit change'' is measurable with respect to $\Sigma$. In our variant, we replace the measurability assumption by a~simpler one: $\Sigma$ should be invariant under both actions and the actions on the resulting factor should be free. In conclusion we provide a characterization of the Pinsker sigma-algebra of any $G$-process in terms of an appropriately defined remote past arising from a multiorder.
	
	The paper has an appendix in which we present an explicit construction of a~particularly regular (uniformly F\o lner) multiorder based on an ordered dynamical tiling system of $G$. 
\end{abstract}

\maketitle

\tableofcontents

\section{Introduction: motivation and organization of the paper}
The additive group $\Z$ of integers has two very important properties associated to the interplay between the order structure and amenability: 
\begin{enumerate}[(1)]
	\item $\Z$ is orderable; the standard order satisfies $n<m\iff n+k<m+k$,
	\item the order intervals $[0,n]$ form a F\o lner \sq\footnote{A F{\o}lner sequence in a group $G$ is a sequence of finite subsets $F_n\subset G$ such that $\lim\limits_{n\to\infty}\frac{|F_n\cap gF_n|}{|F_n|}=1$, for any $g\in G$.}. 
\end{enumerate}
These two properties play a key role in calculating dynamical entropy and characterizing the Pinsker factor of a stationary $\Z$-process. In particular, if $\P$ is a finite measurable partition of a probability space $(X,\Sigma,\mu)$ with an action of a measure-automorphism $T$, then
the following formula holds 
\begin{equation}\label{eq0}
	h(\mu,T,\P)=H(\mu,\P|\P^-),
\end{equation}
where $h(\mu,T,\P)$ is the measure-theoretic entropy of the process generated by the partition $\P$ and $\P^-=\P^{(-\infty,-1]}=\bigvee_{i=1}^\infty T^i(\P)$ is the \emph{past} of this process.\footnote{It might be confusing that the past $\P^-$ depends on the positive iterates $T^i(\P)$.
	It is so, because these \emph{positive} iterates of the partition describe the behavior of the \emph{backward} orbit (i.e.\ the \emph{past}) of points: if $A$ is an atom of $\P$ then 
	\begin{equation*}
		x\in T^i(A) \iff T^{-i}(x)\in A.
	\end{equation*}
	\vspace{-10pt}
} The symbol $H(\mu,\P|\P^-)$ denotes the Shannon (static) entropy of $\P$ conditional with respect to the sigma-algebra $\P^-$. Moreover, the Pinsker sigma-algebra of this process, defined as the largest \inv\ sub-sigma-algebra of $\Sigma$ on which the action has entropy zero, is characterized by the formula
\begin{equation}\label{eq1}
	\Pi_{T}(\P)=\bigcap_n\P^{(-\infty,-n]}.
\end{equation}
The above intersection is often referred to as the \emph{remote past} of the process.

\medskip
Let now $G$ denote a countable group. By a \emph{total order} on $G$ we will mean a transitive relaton $\prec$ such that for every $a,b\in G$ exactly one of the alternatives holds: either $a\prec b$ or $b\prec a$, or $a=b$. Total order on $G$ is \emph{invariant} if the implication $a\prec b\Rightarrow ag \prec bg$ is true for all $a,b,g\in G$. A group admitting an invariant total order is called \emph{orderable}. In general, $G$ need not be orderable, and if it is, the order intervals need not form a F\o lner \sq. For example, $\Z^2$ is orderable but no invariant order has the property (2) (the reader may easily verify that there is no invariant order on $\Z^2$ whose every order interval is finite). 
In~1975, John Kieffer~\cite{K} introduced an interesting substitute of an invariant order, the \emph{invariant random order}. 
\begin{defn}
	An \emph{invariant random order} (IRO) on $G$ is a probability space $(\O,\nu)$, where $\O$ is a measurable family of total orders $\prec$ on $G$ (represented as $\{0,1\}$-valued functions on $G\times G$) and $\nu$ is a Borel probability measure supported by $\O$, invariant under the action of $G$ on $\O$, defined by the rule $(g,\prec)\mapsto\ \prec'$, where
	\begin{equation*}
		a\ \prec'\ b \iff ag\ \prec\ bg \ \ (a,b\in G).
	\end{equation*}
\end{defn}
It is easy to show that IRO exists on any countable group. For instance, take the i.i.d.\ $G$-process $(\mathsf X_g)_{g\in G}$ with values in the interval $[0,1]$ distributed according to the Lebesgue measure. Almost every realization $\omega$ of this process is an injective function $g\mapsto\omega_g$ from $G$ to $[0,1]$. The linear ordering of the values determines a linear ordering $\prec_\omega$ of $G$, as follows:
\begin{equation*}
	a\,\prec_\omega\,b\ \iff\ \omega_a<\omega_b, \ \ (a,b\in G).
\end{equation*}
It is elementary to verify that $\omega\mapsto\ \prec_\omega$ defines a measure-theoretic factor of the i.i.d.\ $G$-process, which is in fact an IRO on $G$ of type $\mathbb Q$ (i.e.\ for $\nu$-almost every $\omega$ one has that for any $a,b\in G$ with $a\prec_\omega b$ there exists $c\in G$ such that $a\prec_\omega c\prec_\omega b$). 

If $G$ is a countable amenable group\footnote{Throughout this paper, by a countable amenable group we will mean an infinite countable discrete group in which there exists a F{\o}lner sequence.}, then so defined IRO, being a factor of a Bernoulli process, has positive entropy (see e.g.\ \cite{OW2}).
Invariant random orders have been successfully applied to the computation of Kolmogorov--Sinai entropy, for proving a version of the Shannon--McMillan--Breiman theorem for actions of countable amenable groups (see~\cite{K}) and further refinements (see e.g. \cite{AMR}).

In this paper we propose a refined version of IRO which we call ``multiorder''. A~multiorder is an IRO with the additional property that all orders $\prec$ in the support of $\nu$ are of type~$\Z$.

\smallskip
Section 2 contains the rigorous definition and several basic facts on multiorders. It turns out that amenability of a countable group $G$ is equivalent to the existence of a multiorder on $G$. Moreover, any such multiorder has the F\o lner property (see Definition~\ref{fp} and Theorem~\ref{TM}), which is a very desirable feature, analogous to the condition (2) listed at the begining of this section in the context of the classical order on $\Z$. 

\smallskip
In Section 3 we show that  multiorders are strongly related to the well-known fact that any measure-preserving action of a countable amenable group $G$ is orbit equivalent to a $\Z$-action (see~\cite{OW}). This relation is captured in our Theorem \ref{motooe}. In particular, this implies that multiorders of entropy zero exist on any countable amenable group.
\smallskip

Another difference between a multiorder and an IRO becomes apparent in their applications to studying measure-preserving $G$-actions. So far (in both papers \cite{K} and \cite{AMR}), an IRO was associated to the group rather than the considered $G$-action. One can say that it played the role of an \emph{external} object. On the contrary, we often assume that a multiorder is a measure-theoretic factor of the $G$-action in question. In fact, we show that if $G$ is amenable, then any free $G$-action\footnote{A measure-theoretic $G$-action $(X,\mu,G)$ is free if the stabilizer of $\mu$-almost every $x\in X$ is trivial, i.e.\ $\{g\in G: g(x)=x\}=\{e\}$.} has a multiorder factor (see Corollary \ref{33}). Thus, multiorder is treated as an \emph{internal} feature of the given $G$-action. Our approach to multiorders leads to more complicated relative results involving disintegration of the measure with respect to the multiorder factor (which is trivial in case of an external IRO). As we shall see, the relative results are quite useful. 

\smallskip

Following these lines, our Section 4 starts with the definition of a ``multiordered $G$-action'', as one equipped with a fixed measure-theoretic multiorder factor. The first main result of the paper is Theorem \ref{newentropy} providing a formula for the entropy relative to the multiorder factor in terms of the random past, analogous to~\eqref{eq0}. The formula reduces to the one given in \cite{AMR} if the multiorder is joined with the action independently (via the product joining). 
\smallskip

In Section 5, building upon Theorem \ref{motooe} (saying, roughly speaking, that a multiordered system is orbit equivalent to a specific $\Z$-action which also factors to the same multiorder, but now regarded with a respective action of $\Z$), we show that the entropy of a multiordered $G$-action conditional given the multiorder is preserved by the orbit equivalence (see Theorem~\ref{eq_entropies}). This is in fact a special case of a more general theorem, due to Rudolph and Weiss (\cite[Theorem 2.6]{RW}), but our proof is totally different.
As a consequence, we derive a surprising fact that, although orbit equivalence usually does not preserve the entropy, the difference between the entropy of a multiordered $G$-action and that of the respective $\Z$-action comes exclusively from changing the dynamics on the multiorder factor (see~Corollary~\ref{surp}).

\smallskip
In Section 6 we show that the mentioned above theorem by Rudolph and Weiss can actually be derived from our Theorem \ref{eq_entropies}. The original theorem states that orbit equivalence between actions of countable amenable groups $G$ and $\Gamma$ preserves the conditional entropy with respect to a $G$-invariant sub-sigma-algebra $\Sigma_Y$, as soon as the actions are free and the ``orbit change'' is measurable with respect to~$\Sigma_Y$. In our variant, we replace the measurability assumption by a simpler one: $\Sigma_Y$ should be invariant under both actions and the actions on the resulting factor $(Y,\nu)$ should be free. Although formally our assumption is slightly stronger, it is technically simpler and easier to check. Also, our proof is completely different and much shorter. In this manner, multiorders emerge as a useful tool in giving new proofs of advanced facts about orbit equivalence.

\smallskip
Finally, in Section 7, we focus on the Pinsker factor of a multiordered $G$-action. It follows immediately from the results of Section 5, that the Pinsker factor relative to the multiorder is preserved by the orbit equivalence to the respective $\Z$-action (Theorem \ref{pinsker_mult}). In this section, however, the best results are obtained for an arbitrary (i.e.\ not necessarily multiordered) $G$-action, by passing to the product joining with a multiorder (which is a multiordered $G$-action). In some sense, this takes us back to treating multiorder as an external object, but our methods depend on the relative results of Section~5. It is so because in the proof we pass to the $\Z$-action orbit-equivalent to the product $G$-action and this $\Z$-action is no longer a product joining; it has the structure of a skew product. In our concluding Theorem~\ref{R-s} we prove that the (unconditional) Pinsker sigma-algebra of any $G$-process can be identified as follows: a set $A$ is measurable with respect to the Pinsker sigma-algebra of a $G$-process if and only if, for some (equivalently any) multiorder $(\tilde{\O},\nu,G)$, $A$ is measurable with respect to the remote past evaluated along $\nu$-almost every order. This theorem sheds a new light on the Pinsker factor even in case of a classical $\Z$-process. The only known proof of the fact that the remote past and the remote future of a $\Z$-process are the same invariant sigma-algebra relies heavily on entropy theory. Up to date no purely measure-theoretic proof has been found (i.e.\ a proof based exclusively on the analysis of the sigma-algebras). Until such a proof is found, one cannot claim that we fully understand this phenomenon. Note that the remote future becomes the remote past if we replace the standard order on $\Z$ by its reverse. Our Theorem \ref{R-s} (although it does not bring us any closer to finding a measure-theoretic proof) makes the mystery even more puzzling: there exists a vast collection of non-standard multiorders on $\Z$ (see e.g.\ Example \ref{Z-ord}) which allow to identify the Pinsker factor.

\medskip

The paper has an Appendix devoted to a brief summary of the theory of tilings of amenable groups and in which we introduce the notion of an  \emph{ordered tiling system}. Next we provide an effective construction of a multiorder arising from an ordered tiling system. We show that the resulting multiorder always enjoys a stronger version of the F\o lner property which we call \emph{uniform F\o lner property}. In this manner we prove a strengthening of Theorem \ref{exi} (Corollary~\ref{muo1}) which asserts that on any countable amenable group there exists a uniformly F\o lner multiorder of entropy zero. At the same time, we demonstrate that the existence of multiorders on a countable amenable group can be viewed as a phenomenon independent of orbit equivalence to $\Z$-actions.
\smallskip

The research of this paper is continued in \cite{DW}, where multiorders play a crucial role in the study of asymptotic pairs and their relation to entropy in topological $G$-actions and lead to establishing further analogs of the results known for $\Z$-actions.
\section{The concept of a multiorder}
\begin{defn}\label{oz}
	Let $G$ be an infinite countable set. A linear order $\prec$ on $G$ is \emph{of type} $\Z$ if every order interval $[a,b]^{\prec}=\{a,b\}\cup\{g\in G: a\prec g\prec b\}$ ($a,b\in G,\ a\prec b$) is finite, and there are no minimal or maximal elements in $G$. In other words, $(G,\prec)$ is order isomorphic to $(\Z,<)$.
\end{defn}
If $\prec$ is an order of type $\Z$, and $[a,b]^\prec$ is an order interval, then by $|[a,b]^\prec|$ we will denote its cardinality and call it the \emph{length} of $[a,b]^\prec$.
\smallskip

The set $\tilde\O$, of all orders of $G$ of type $\Z$ is a subset of the family of all relations on~$G$, which in turn can be identified with $\{0,1\}^{G\times G}$. Thus, $\tilde\O$ inherits from $\{0,1\}^{G\times G}$ a natural (metrizable and separable) \tl\ structure. By an easy proof, $\tilde\O$ is a nonempty measurable subset of $\{0,1\}^{G\times G}$.

\begin{defn}\label{mo}
	Let $G$ be an infinite countable group. This group acts on $\tilde\O$ by homeomorphisms as follows: for $\prec\ \in\tilde\O$ and $g\in G$ the image $\prec'\,=g(\prec)$ is given by 
	\begin{equation}\label{ac}
		a\prec' b\iff ag\prec bg
	\end{equation}
	(it is elementary to see that $\prec'$ is again an order of type $\Z$).
\end{defn}

We are in a position to introduce the key notion of this paper. 

\begin{defn}
	Let $\nu$ be a $G$-\inv\ Borel probability measure supported by $\tilde\O$. By a \emph{multiorder} (on $G$) we mean the measure-preserving $G$-action $(\tilde\O,\Sigma_{\tilde\O},\nu,G)$, where $\Sigma_{\tilde\O}$ is the Borel sigma-algebra on $\tilde\O$.
\end{defn}

For brevity, from now on, we will skip the indication of sigma-algebras in the notation of measure-preserving actions. For example, a multiorder will be denoted by $(\tilde\O,\nu,G)$
or just $\tilde\O$, when this does not lead to a confusion.

A priori, it is not clear for what kind of countable groups an \im\ on $\tilde\O$ exists. Our Theorems \ref{TM} and \ref{exi} imply that the existence of such a measure is in fact equivalent to amenability (see Corollary \ref{char}).

In case the group $G$ is amenable it is natural to consider multiorders with the following  additional property:

\begin{defn}\label{fp}
	Let $G$ be a countable amenable group with the unit denoted by~$e$. A~multiorder $(\tilde\O,\nu,G)$ has the \emph{F\o lner property} (or, briefly, is \emph{F\o lner}) if, for $\nu$-almost every $\prec\ \in\tilde\O$, the order intervals $[e,b_n]^\prec$ (where $b_n$ denotes the $n$th successor of $e$ in the order $\prec$) form a F\o lner \sq\ in $G$.
\end{defn}

It turns out that the F\o lner property is automatic, as the following theorem states. The proof, which was suggested by Tom Meyerovitch, and which relies on an orbit-equivalence to a $\Z$-action, is provided in the next section. Alexandre Danilenko pointed out that the same result can be alternatively derived from \cite[Lemma 3.10]{RW} using the Borel--Cantelli Lemma.
\begin{thm}\label{TM}{\rm(Suggested by Tom Meyerovitch)}
	Let $G$ be any countable group and assume that there exists a multiorder $(\tilde \O, \nu, G)$ on $G$. Then the multiorder $(\tilde \O,\nu,G)$ has the F\o lner property, in particular $G$ is amenable.
\end{thm}
The following theorem will be proved in Section \ref{trzy} (see Corollary \ref{34}; in~the Appendix the reader will find a strengthened version, Corollary \ref{muo1}).

\begin{thm}\label{exi}
	Let $G$ be a countable amenable group. There exists a multiorder $(\tilde\O,\nu,G)$ of entropy zero. 
\end{thm}

\begin{cor}\label{char}
	A countable group $G$ is amenable if and only if there exists a multiorder on $G$.
\end{cor}

\begin{cor} In every countable amenable group $G$ there exists a F\o lner \sq\ $\F=(F_n)_{n\in\N}$ with the following properties
	\begin{itemize}
		\item $F_n\subset F_{n+1}$, $n\in\N$ (i.e.\ $\F$ is nested),
		\item $F_1=\{e\}$ (i.e.\ $\F$ is centered),
		\item $|F_n|=n$ (i.e.\ $\F$ ``progresses by one element'').
	\end{itemize}
\end{cor}
\begin{proof}
	Let $(\tilde{\O},\nu,G)$ be a multiorder on $G$ whose existence is guaranteed by
	Theorem~\ref{exi}. Then, by Theorem \ref{TM}, a desired F\o lner \sq\ is obtained by fixing a $\nu$-typical order $\prec\ \in\tilde\O$ and letting $F_n$ be the order-interval of length $n$ starting~at~$e$.
\end{proof}

For any countable group $G$, there are at least two other ways of representing the action of $G$ on the family $\tilde\O$ of all orders of type $\Z$ in a symbolic form. The first one refers to the concept of an increment:

\begin{defn}
	Let $\prec\ \in\tilde\O$ be an order of type $\Z$ of a countable group $G$. For each $g$ let $\mathsf{succ}_\prec(g)$ denote the successor of $g$ with respect to $\prec$ and define \emph{the (left) increment} at $g$ as $\incr_\prec(g)=\mathsf{succ}_\prec(g)g^{-1}$. Let $\pi:\tilde\O\to G^G$ be given by $\prec\ \mapsto \incr_\prec$, where
	$\incr_\prec$ stands for the function $\incr_\prec(\cdot)$.
\end{defn}

We have the following:
\begin{prop}\label{odw}
	The mapping $\pi$ is a measurable injection\footnote{Here we mean a Borel measurable injection. Recall that Borel-measurable injection sends Borel sets to Borel sets, i.e.\ the inverse map is Borel-measurable, see \cite[Theorem~15.1 and Corollary~15.2]{Ke}.}, and it intertwines the action of $G$ on $\tilde\O$ defined by \eqref{ac} with the shift action\footnote{For any set $\Lambda$, the shift action of $G$ on $\Lambda^G$ is defined as follows: $(g(x))_{h}=x_{hg}$, where $g\in G$ and $x=(x_h)_{h\in G}\in \Lambda^G$.} on the image $\pi(\tilde\O)\subset G^G$.
\end{prop}

\begin{proof}
	Injectivity is obvious: two orders of type $\Z$, say $\prec$ and $\prec'$, are different if and only if they assign different successors, $\mathsf{succ}_\prec(g)$ and $\mathsf{succ}_{\prec'}(g)$, to at least one $g\in G$. Then the respective increments, $\incr_\prec(g)$ and $\incr_{\prec'}(g)$ are also different.
	
	For measurability of $\pi$ it suffices to show that for any $a,b\in G$ the set
	\begin{equation*}
		\{\prec\ \in\tilde O: \incr_\prec(a)=b\}
	\end{equation*}
	is measurable. This set can be written as $\{\prec\ \in\tilde O: \mathsf{succ}_\prec(a)=ba\}$, which in turn equals 
	\begin{equation*}
		\{\prec\ \in\tilde O: a\prec ba \text{ and }(\nexists c\in G)\, a\prec c\prec ba\}.
	\end{equation*}
	The above set is easily seen to be a countable combination of unions and intersections of sets depending on just one binary order relation, which are obviously closed~in~$\tilde\O$.
	
	It follows from \eqref{ac} that 
	\begin{equation}\label{su}
		b=\mathsf{succ}_{g(\prec)}(a)\iff bg = \mathsf{succ}_\prec(ag).
	\end{equation}
	Then 
	\begin{equation*}
		\incr_{\prec}(ag)=bg(ag)^{-1}=ba^{-1}=\incr_{g(\prec)}(a).
	\end{equation*} 
	The left hand side, $\incr_{\prec}(ag)$, equals $(g(\incr_\prec))(a)$, where $g(\cdot)$ denotes the shift action of $g$~on~$G^G$. We have proved that the following diagram commutes:
	\begin{equation*}
		\begin{CD}
			\prec @>g>\eqref{ac}> g(\prec)\\
			@V\pi VV @VV\pi V\\
			\incr_\prec @>g>\text{(shift)}> g(\incr_\prec)=\incr_{g(\prec)},
		\end{CD}
	\end{equation*}
	which ends the proof.
\end{proof}

Another approach to the action of $G$ on $\tilde\O$ relies on representing orders of type $\Z$ as bijections from $\Z$ to $G$.
\begin{defn}\label{ddd}
	With each order $\prec\ \in\tilde\O$ we associate the bijection $\mathsf{bi}_\prec:\Z\to G$ which is \emph{anchored} (i.e.\ satisfies $\mathsf{bi}_\prec(0)=e$), and on the rest of $\Z$ is determined by the property:
	\begin{equation}\label{bsu}
		\mathsf{bi}_\prec(i) = g \iff \mathsf{bi}_\prec(i+1)= \mathsf{succ}_\prec(g), \ 
		i\in\Z,\ g\in G.
	\end{equation}
	
	On anchored bijections $\mathsf{bi}:\Z\to G$ we define the action of $G$ by the following formula:
	\begin{equation}\label{bij}
		(g(\mathsf{bi}))(i)= \mathsf{bi}(i+k)\cdot g^{-1}, \text{ where $k$ is such that } g=\mathsf{bi}(k).
	\end{equation}
\end{defn}

The space $\mathbf{Bi}(\Z,G)$ of all anchored bijections from $\Z$ to $G$ equipped with the natural topological structure inherited from $G^{\Z}$ is clearly Borel-measurable (in fact this set is of type $G_{\delta}$, hence it is a Polish space). The fact that \eqref{bij} defines an action will follow from Proposition \ref{odw1}.

\begin{prop}\label{odw1} The assignment $\rho:\tilde \O\to \mathbf{Bi}(\Z,G)$, given by 
	$\rho(\prec)=\mathsf{bi}_\prec$, is a measurable bijection with a continuous inverse, which intertwines the action of $G$ on $\tilde\O$ given by \eqref{ac} with the action of $G$ on $\mathbf{Bi}(\Z,G)$ given by~\eqref{bij}. 
\end{prop}

\begin{proof} First of all, note that for any $\mathsf{bi}\in\mathbf{Bi}(\Z,G)$ and any $g\in G$, by \eqref{bij}, we have $(g(\mathsf{bi}))(0)=\mathsf{bi}(k)\cdot g^{-1}=gg^{-1}=e$ (where $k$ is such that $g=\mathsf{bi}(k)$), so $g(\mathsf{bi})$ is anchored. Injectivity of the assignment $\prec\ \mapsto\mathsf{bi}_\prec$ is obvious, and so is its surjectivity: every anchored bijection $\mathsf{bi}:\Z\to G$ naturally defines on $G$ an order $\prec$ of type $\Z$ such that $\mathsf{bi}=\mathsf{bi}_\prec$. The inverse assignment $\mathsf{bi}\mapsto\ \prec$ is continuous because any finite set $K\subset G$ is contained in the image $\mathsf{bi}([-n,n])$ for large enough $n$ and so $\mathsf{bi}([-n,n])$ determines the order between the elements of $K$ (this argument fails for the map $\rho$; for instance there is no finite set $K$ such that $\mathsf{bi}_\prec(1)$ can be determined based on the order $\prec$ just between the elements of $K$). Measurability of $\rho$ follows from bijectivity and measurability of the inverse (see the footnote on the preceding~page).
	
	To complete the proof we need to show that the following diagram commutes:
	\begin{equation*}
		\begin{CD}
			\prec @>g>\eqref{ac}> g(\prec)\\
			@VVV @VVV\\
			\mathsf{bi}_\prec @>g>\eqref{bij}> g(\mathsf{bi}_\prec)=\mathsf{bi}_{g(\prec)},
		\end{CD}
	\end{equation*}
	i.e.\ that $(g(\mathsf{bi}_\prec))(i)=\mathsf{bi}_{g(\prec)}(i)$, for all $i\in\Z$. For $i=0$ this follows from both $\mathsf{bi}_{g(\prec)}$ and $g(\mathsf{bi}_\prec)$ being anchored.
	Choose $i\!>\!0$ and let $h=\mathsf{bi}_{g(\prec)}(i)$ (we omit the similar case $i\!<\!0$). Then $[e,h]^{g(\prec)}$ is an order interval (with respect to $g(\prec)$) of length $i+1$. According to \eqref{ac}, $[g,hg]^\prec$ is an order interval of length $i+1$ with respect to~$\prec$\,. This implies that if $k$ is such that $\mathsf{bi}_{\prec}(k)=g$, then $hg=\mathsf{bi}_{\prec}(i+k)$, i.e.\ $h=\mathsf{bi}_{\prec}(i+k)\cdot g^{-1}$. By \eqref{bij}, the latter expression equals
	$(g(\mathsf{bi}_{\prec}))(i)$, and we are done. \end{proof}

We will use the following notation: given an order $\prec\ \in\tilde\O$ we abbreviate $\mathsf{bi}_\prec(i)$ as $i^\prec$ ($i\in\Z$). In particular $0^{\prec}=e$ regardless of $\prec$. Given $i<j\in\Z$, the order interval $[i^\prec,j^\prec]^\prec\,=\{i^\prec,(i+1)^\prec,(i+2)^\prec,\dots,j^\prec\}$ will be denoted by $[i,j]^\prec$. Analogous notation $[i,\infty)^\prec$ and $(-\infty, j]^\prec$ will be used for unbounded intervals. Given $g\in G$ and $n\in\N$, by $[g,g+n]^\prec$ (resp. $[g-n,g]^\prec$) we will denote the order interval of length $n+1$ starting (resp. ending) at $g$ (for instance, $[e,e+n]^\prec\,=[0,n]^\prec$). Also, for $F\subset G$, we let \begin{equation}
	[F,F+n]^\prec\,=\bigcup_{g\in F}\,[g,g+n]^\prec\ \text{ and }\ [F-n,F]^\prec\,=\bigcup_{g\in F}\,[g-n,g]^\prec.
\end{equation} 

In this notation, the formula \eqref{bij} defining the action of $G$ on anchored bijections takes on the following form:
\begin{equation}\label{16}
	i^{g(\prec)} = (i+k)^{\prec}\cdot g^{-1}, \text{ \ equivalently \ }i^{\prec}\cdot g^{-1}= (i-k)^{g(\prec)},
\end{equation}
where $k$ is the unique integer such that 
\begin{equation}\label{17}
	g=k^\prec, \text{ \ equivalently \ } g^{-1}=(-k)^{g(\prec)}.
\end{equation}

\begin{rem}
	In general (unlike in the case of $G=\Z$), $(k^\prec)^{-1}$ does not equal $(-k)^\prec$. 
\end{rem}

Summarizing this section, we have introduced three isomorphic $G$-actions: the action on the set $\tilde\O$ of all orders of type~$\Z$, given by \eqref{ac}, the action of the usual\break$G$-shift on $G^G$ restricted to $\pi(\tilde\O)$, and the action given by \eqref{bij} on anchored bijections. Thus, if $G$ is amenable, a multiorder can be understood in three equivalent ways: as a measure-theoretic $G$-action with a Borel invariant probability measure~$\nu$ on $\tilde\O$, the action of the $G$-shift on $G^G$ with the measure $\pi(\nu)$, and as the action given by \eqref{bij} on $\mathbf{Bi}(\Z,G)$ equipped with the measure $\rho(\nu)$. For simplicity, in all three cases we will denote the measure as $\nu$. In the sequel, we will mainly use the first and last representations.

\section{Multiorder and orbit equivalence to an action of the integers}\label{trzy}
Recall that two measure-theoretic group actions, say $(X,\mu,\Gamma)$ and $(Y,\nu,G)$, where both groups $\Gamma$ and $G$ are countable, are \emph{orbit equivalent} if there exists a~measure-automorphism $\psi:(X,\mu)\to (Y,\nu)$ which sends $\Gamma$-orbits to $G$-orbits, i.e.\ for $\mu$-almost every $x\in X$ and every $\gamma\in\Gamma$ there exists $g_{x,\gamma}\in G$ such that 
\begin{equation*}
	\psi(\gamma(x))=g_{x,\gamma}(\psi(x)),
\end{equation*}
and $\{g_{x,\gamma}: \gamma\in\Gamma\}=G$.
We can also write
\begin{equation*}
	\gamma(x)=\psi^{-1}g_{x,\gamma}\psi(x),
\end{equation*}
which means that the identity map establishes an orbit equivalence between the given $\Gamma$-action on $(X,\mu)$ and the $G$-action on $(X,\mu)$ given by 
\begin{equation*}
	g(x) = \psi^{-1}g\psi(x),
\end{equation*}
which is obviously isomorphic to the original action of $G$ on $(Y,\nu)$.
This reduces the considerations of orbit equivalent actions to actions defined on the same probability space $(X,\mu)$ and such that the orbit equivalence is established by the identity map (i.e.\ both actions have the same orbits). In such case, for $\mu$-almost every $x\in X$ we have a relation $\mathsf R_x$ between the elements of $\Gamma$ and $G$:
\begin{equation*}
	\gamma\ \mathsf R_x\ g\iff \gamma(x)=g(x),
\end{equation*}
such that $\mathsf R_x$ has full projections on $\Gamma$ and on $G$.
In case when both actions are free, the above relation is a bijection, and we can write
\begin{equation}\label{bibi}
	g=\mathsf{bi}_x(\gamma) \iff \gamma(x)=g(x).
\end{equation}
In this case, $x\mapsto\mathsf{bi}_x$ is a measurable assignment from $X$ to the collection $\mathbf{Bi}(\Gamma,G)$ of all anchored (i.e.\ sending the unit of $\Gamma$ to the unit of $G$) bijections from $\Gamma$ to $G$, called a \emph{cocycle}.
\smallskip

Another way of viewing orbit equivalence is as follows: 
Recall that the \emph{full group} of an action of a countable group $\Gamma$ on a measure space $(X,\mu)$ consists of all measurable invertible transformations $T:X\to X$ such that for $\mu$-almost every $x\in X$ one has $T(x)=\gamma_x(x)$ for some $\gamma_x\in\Gamma$. It is well known that if the action of $\Gamma$ preserves the measure $\mu$ then every $T$ in the full group also preserves $\mu$. We have the following, almost obvious, fact (whose proof we~skip):
\begin{fact}
	Two actions $(X,\mu,\Gamma)$ and $(X,\mu,G)$ are orbit equivalent via the identity map if and only if for every $g\in G$ the transformation $x\mapsto g(x)$ belongs to the full group of the $\Gamma$-action and for every $\gamma\in\Gamma$ the transformation $x\mapsto \gamma(x)$ belongs to the full group of the $G$-action.
\end{fact}

Orbit equivalence is connected to our main topic -- the multiorder -- via the following two theorems. Notice that for a $\Z$-action to be free it suffices that almost all orbits are infinite. Hence any $\Z$-action that is orbit equivalent to a free $G$-action is also free and the assignment $x\mapsto \mathsf{bi}_x$ given by \eqref{bibi} is well defined.
\begin{thm}\label{oetomo}
	Let $(X,\mu, G)$ be a free, measure-preserving\/ $G$-action on a probability space. Let $T:X\to X$ be a measure-automorphism which generates a $\Z$-action orbit equivalent (via the identity map) to $(X,\mu,G)$. Then the map $x\mapsto\mathsf{bi}_x$ given by \eqref{bibi} is a measure-theoretic factor map from $(X,\mu,G)$ to a multiorder $(\tilde\O,\nu,G)$, where $\nu$ is the image of $\mu$ by the above map, and the action of $G$ on $\tilde\O$ is given by the formula \eqref{bij}.
\end{thm}
Before the proof we draw some corollaries, important for the rest of this paper. 
\begin{cor}\label{33}
	Any free measure-preserving action $(X,\mu,G)$ of a countable amen\-able group $G$ on a probabi\-lity space has a multiorder as a measure-theoretic factor.
\end{cor}
\begin{proof}
	It is well known that any measure-preserving (in fact, any nonsingular) $G$-action on a probability space is orbit-equivalent to a $\Z$-action (see \cite[Theorem~6]{OW}). If the $G$-action is free, so is the $\Z$-action\footnote{In general a $\Gamma$-action orbit-equivalent to a free $G$-action need not be free, so the comment preceding Theorem \ref{oetomo} is essential.}, and Theorem \ref{oetomo} applies.
\end{proof}

\begin{cor}\label{34}{\rm(Theorem \ref{exi})}
	On each countable amenable group there exists a multiorder $(\tilde{\O},\nu,G)$ of entropy zero.
\end{cor}
\begin{proof}
	Apply Corollary \ref{33} to a free zero entropy action of $G$ (for the existence of such an action see e.g. \cite[Thoerem 6.1]{DHZ}). 
\end{proof}

\begin{proof}[Proof of Theorem \ref{oetomo}]By \eqref{bibi}, and since both actions are free, the bijections $\mathsf{bi}_x$ are anchored for $\mu$-almost all $x\in X$. The only thing requiring a proof is the equivariance of the map $x\mapsto\mathsf{bi}_x$. We need to show that, for every $g\in G$ and $i\in\Z$ and for $\mu$-almost every $x\in X$, the following equality holds:
	\begin{equation*}
		\mathsf{bi}_{g(x)}(i)=(g(\mathsf{bi}_x))(i),
	\end{equation*}
	where by the formula \eqref{bij},
	\begin{equation*}
		(g(\mathsf{bi}_x))(i)=\mathsf{bi}_x(i+k)\cdot g^{-1}
	\end{equation*}
	with $k$ such that $g=\mathsf{bi}_x(k)$.
	By \eqref{bibi} and since the actions are free, the elements $h=\mathsf{bi}_{g(x)}(i)$ and $h'= \mathsf{bi}_x(i+k)$ are ($\mu$-almost surely) the unique members of $G$ for which the respective equalities hold: 
	\begin{enumerate}[(1)]
		\item $T^i(g(x))=hg(x)$, 
		\item $T^{i+k}(x)=h'(x)$, 
	\end{enumerate}
	while the fact that $g=\mathsf{bi}_x(k)$ means that 
	\begin{enumerate}
		\item[(3)]$g(x)=T^k(x)$. 
	\end{enumerate}
	Combining (1) and (3) we get $T^{i+k}(x)=hg(x)$, which, combined with (2) yields $h'(x)=hg(x)$. Because the action of $G$ is free, for $\mu$-almost every $x$ the last equality allows to  conclude that $h'=hg$, i.e.\ $h'g^{-1}=h$, which is exactly what we needed to show.
\end{proof}

\begin{thm}\label{motooe}
	Suppose $\varphi:X\to\tilde\O$ is a measure-theoretic factor map from a~measure-preserving $G$-action $(X,\mu,G)$ to a multiorder $(\tilde\O,\nu,G)$. Then $(X,\mu,G)$ is orbit-equivalent (via the identity map) to the $\Z$-action generated by the \emph{successor map} defined as follows:
	\begin{equation}\label{full}
		S(x) = 1^\prec(x), \text{ \ where \ }\prec\ = \varphi(x),
	\end{equation}
	i.e.\ $S(x)=g(x)$, where $g=1^{\varphi(x)}=1^{\prec}$.
	Moreover, for any $k\in\Z$, we have
	\begin{equation}\label{TD}
		S^k(x)=k^\prec(x).
	\end{equation}
	Let $\tilde{S}$ denote the transformation on $\tilde{\O}$ defined by
	\begin{equation}\label{SforO}
		\tilde{S}(\prec)=1^{\prec}(\prec),
	\end{equation}
	i.e.\ $\tilde{S}(\prec)=g(\prec)$, where $g=1^{\prec}$ and $g(\prec)$ is given by the formula \eqref{ac}.
	Then $\tilde{S}$ preserves the measure $\nu$ and $\varphi$ is a factor map from the $\Z$-action $(X,\mu,S)$ to the $\Z$-action $(\tilde{O},\nu,\tilde{S})$.
\end{thm}

\begin{rem}
	Note that we do not assume the $G$-actions on $X$ or on $\tilde\O$ to be free.
\end{rem}
\begin{rem}
	In view of the Dye Thoerem (see \cite{Dy}), the orbit-equivalence part of Theorem \ref{motooe} is seemingly trivial for ergodic actions of countable amenable groups. What is special about the action generated by the transformation~$S$ is that it preserves the multiorder factor and, as will be shown in Section~\ref{siedem}, it also preserves the corresponding conditional entropy.
\end{rem}
\begin{rem}\label{pooetomo}
	Let $(X,\mu,G)$ be a free, measure-preserving $G$-action with the same orbits as a $\Z$-action $(X,\mu,T)$. By Theorem \ref{oetomo}, there exists a factor map ${\varphi: X\to \tilde{\O}}$, where $(\tilde{\O},\nu,G)$ is some multiorder. Now, Theorem \ref{motooe} asserts that $(X,\mu,G)$ has the same orbits as the $\Z$-action $(X,\mu,S)$, where $S$ is the successor map given by the formula \eqref{full}. It is not hard to verify (by combining the formulae \eqref{bibi} and~\eqref{full}) that in this case, the maps $S$ and $T$ coincide.
\end{rem}

\begin{cor}\label{35}
	On each countable amenable group $G$ there exists a multiorder $(\tilde{\O},\nu,G)$ of ``double entropy zero'', meaning that
	$h(\nu,G)=h(\nu,\tilde S)=0$, where $h(\nu,G)$ denotes the Kolmogorov--Sinai entropy of $\nu$ under the action of $G$, while $h(\nu,\tilde S)$ denotes the Kolmogorov--Sinai entropy of $\nu$ under the action of $\Z$ by the iterates of $\tilde S$.
\end{cor}

\begin{proof}
	We start by selecting an ergodic free zero entropy action $(X,\mu,G)$ (whose existence follows by the same argument as in the proof of Corollary~\ref{34}). Since $(X,\mu,G)$ is orbit equivalent to a $\Z$-action and, by the Dye Theorem \cite{Dy}, all ergodic $\Z$-actions are mutually orbit equivalent, there exists a $\Z$-action 
	$(X,\mu,T)$ of entropy zero having the same orbits as $(X,\mu,G)$. By Theorem \ref{oetomo}, there exists a multiorder $(\tilde O,\nu,G)$ which is a factor of $(X,\mu,G)$ (hence has entropy zero) and such that $(\tilde O,\nu,\tilde S)$ is a factor of $(X,\mu,S)$. By Remark \ref{pooetomo}, $S=T$, hence $(X,\mu,S)$ has entropy zero. This implies that $(\tilde O,\nu,\tilde S)$ has entropy zero, as required.
\end{proof}

\begin{proof}[Proof of Theorem \ref{motooe}]
	Clearly, the map $S:X\to X$ defined by \eqref{full} is measurable. For the orbit equivalence between $(X,\mu,G)$ and $(X,\mu,S)$ it suffices to prove \eqref{TD} for $S$ defined by \eqref{full}. Indeed, since $k^\prec$ (with $k\in\Z$) ranges over the entire group~$G$, \eqref{TD} implies that the orbits $\{S^k(x):k\in\Z\}$ and $G(x)=\{g(x):g\in G\}$ are equal.
	We will first show \eqref{TD} for $k\ge 0$, by induction. 
	Clearly, \eqref{TD} is true for $k=0$ and, by \eqref{full}, for $k=1$. Suppose it holds for some $k\ge 1$. Then
	\begin{multline*}
		S^{k+1}(x) = S(k^\prec(x))= 1^{\varphi(k^\prec(x))}(k^\prec(x))=1^{k^\prec(\varphi(x))}(k^\prec(x))=\\1^{k^\prec(\prec)}(k^\prec(x))=
		1^{g(\prec)}(g(x)),
	\end{multline*}
	where $g=k^\prec$. By \eqref{16} and \eqref{17} (applied to $i=1$), we have
	\begin{equation*}
		1^{g(\prec)}=(k+1)^\prec g^{-1}.
	\end{equation*}
	Eventually, $S^{k+1}(x)=(k+1)^\prec g^{-1}(g(x))=(k+1)^\prec(x)$,
	and \eqref{TD} is shown for $k+1$. 
	
	Now consider the map $U(x)=(-1)^\prec(x)$, where, as before, $\prec\ =\varphi(x)$. By an inductive argument analogous as that used for $S$, one can show that ($\mu$-almost surely) for any $k\ge 0$ the following holds:
	\begin{equation*}
		U^k(x) = (-k)^\prec(x).
	\end{equation*}
	We will show that $U$ is the inverse map of $S$. Denote $x'=U(x)$ and $\prec'\ = \varphi(x')$. Then we have $x=((-1)^{\prec})^{-1}(x')$, and
	\begin{equation*}
		\prec'\ = \varphi(x')=\varphi(U(x))=\varphi((-1)^\prec(x))=(-1)^\prec(\varphi(x))=(-1)^\prec(\prec),
	\end{equation*}
	in other words,
	\begin{equation*}
		\prec\ = ((-1)^\prec)^{-1}(\prec').
	\end{equation*}
	
	By the second part of \eqref{17} applied to $g=(-1)^{\prec}$, we have
	\begin{equation*}
		((-1)^{\prec})^{-1}=1^{g(\prec)} =1^{(-1)^{\prec}(\prec)}=1^{\prec'},
	\end{equation*}
	and hence,
	\begin{equation*}
		x=1^{\prec'}(x')=1^{\varphi(x')}(x')=S(x') = S(U(x)).
	\end{equation*}
	By a symmetric argument we also have $x= U(S(x))$, which implies, on one hand, that $S$ is invertible (with the inverse $U$), and on the other hand, that \eqref{TD} holds for negative integers. This ends the proof of the orbit equivalence between $(X,\mu,G)$ and $(X,\mu,S)$.
	
	Since $(\tilde{\O},\nu,G)$ is a factor of itself, the first part of Theorem \ref{motooe} can be applied to $(\tilde{\O},\nu,G)$ in place of $(X,\mu,G)$. This implies that $(\tilde{\O},\nu,\tilde S)$ is orbit equivalent to $(\tilde{\O},\nu,G)$, in particular, $\tilde S$ preserves $\nu$. For $\prec\ =\varphi(x)$ we have
	\begin{equation*}
		\varphi(S(x))=\varphi(1^{\prec}(x))=1^{\prec}(\varphi(x))=1^{\prec}(\prec)=\tilde S(\varphi(x)),
	\end{equation*}
	where the central equality follows from the fact that $\varphi$ commutes with all elements of $G$. We have shown that $\varphi$ is a factor map between the $\Z$-actions $(X,\mu,S)$ and $(\tilde{\O},\nu,\tilde S)$, which completes the proof.
\end{proof}
\begin{rem}
	As pointed out by Alexandre Danilenko \cite{D0}, Theorem~\ref{motooe} can be alternatively proved basing on a more general approach developed in \cite{D1}.
\end{rem}

We pass to proving Theorem \ref{TM}. The proof is preceded by a lemma concerning amenable group $\Gamma$. Nonetheless, in the proof of Theorem~\ref{TM}, the lemma will be applied to $\Gamma=\Z$, not to the general group $G$. 

\begin{lem}\label{tom}
	Let $(X,\mu,\Gamma)$ be an ergodic measure-preserving action of a countable amenable group $\Gamma$ and let $(F_n)_{n\ge1}$ be a F\o lner \sq\ in $\Gamma$ along which the pointwise ergodic theorem holds.\footnote{For instance a \emph{tempered} F\o lner \sq, see \cite{L} for details.} Let $T:X\to X$ be a member of the full group of this action. Then, for $\mu$-almost every $x\in X$, any $\eps>0$ and $n$ sufficiently large, the following inequality holds:
	\begin{equation*}
		|F_n(x)\ \triangle\ T(F_n(x))|\le\eps|F_n|.
	\end{equation*}
\end{lem}

\begin{proof}The sets $X_\gamma=\{x\in X: T(x)=\gamma(x)\}$ ($\gamma\in\Gamma$) form a countable, measurable cover of $X$ (if the action of $\Gamma$ on $X$ is free, it is a partition). Thus, there exists a~finite set $K\subset\Gamma$ such that $\mu(X_K)>1-\frac\eps4$, where $X_K=\bigcup_{\gamma\in K}X_\gamma$. Let 
	\begin{equation*}
		F_n'=\{\beta\in F_n: \beta(x)\in X_K\}.
	\end{equation*} By the pointwise ergodic theorem, for $\mu$-almost every $x\in X$, for $n$ large enough, we have
	$|F'_n|>(1-\tfrac\eps4)|F_n|$.
	Note that $T(F_n'(x))\subset KF_n(x)$. Additionally, if $n$ is large enough then $F_n$ is $(K,\frac\eps4)$-\inv\footnote{A set $A\subset G$ is $(g,\varepsilon)$-invariant for some $g\in G$, if $\frac{|A\triangle gA|}{|A|}<\varepsilon$, where $\triangle$ denotes the symmetric difference. Similarily, for $K\subset G$, the set $A$ is $(K,\varepsilon)$-invariant if $\frac{|A\triangle KA|}{|A|}<\varepsilon$.}. Hence, we can write
	\begin{multline*}
		|T(F_n(x))\setminus F_n(x)|\le |T(F_n'(x))\setminus F_n(x)| + |T(F_n(x))\setminus T(F'_n(x))|\le\\|(KF_n\setminus F_n)(x)|+|(F_n\setminus F'_n)(x)|\le 
		|KF_n\setminus F_n|+|F_n\setminus F'_n|\le \tfrac\eps4|F_n|+\tfrac\eps4|F_n|=\tfrac\eps2|F_n|.
	\end{multline*}
	Since $T$ is invertible, the sets $F_n(x)$ and $T(F_n(x))$ have equal cardinalities, and hence the other difference $F_n(x)\setminus T(F_n(x))$ has the same cardinality, just shown to be smaller than $\frac\eps2|F_n|$. Thus the symmetric difference has cardinality less than $\eps|F_n|$ and the proof is completed.
\end{proof}

\begin{proof}[Proof of Theorem \ref{TM}] Let $(\tilde{\O},\nu,G)$ be a multiorder on a countable group~$G$. By a standard ergodic decomposition argument, it suffices to prove the theorem in case $\nu$ is ergodic. There exists a free ergodic action $(X,\mu,G)$ which has the multiorder $(\tilde\O,\nu,G)$ as a factor via a factor map $\varphi$ (for instance, take an ergodic joining\footnote{A joining of two measure-preserving systems, say $(X,\mu, G)$ and $(Y,\nu,G)$ is any measure on $X\times Y$ (usually denoted as $\mu\vee\nu$) invariant under the product action of $G$, and whose respective marginals are $\mu$ and $\nu$. It is well known that if both $\mu$ and $\nu$ are ergodic then there exists an ergodic joining $\mu\vee\nu$ (see e.g.\ \cite[Proposition~1.4]{TR}; the same proof applies to general countable group actions).} of the multiorder with a Bernoulli $G$-process, and $\varphi$ being the projection on the first coordinate). 
	Let $S:X\to X$ be given by~\eqref{full}. Then, any $g\in G$, viewed as a~transformation of $X$, belongs to the full group of the $\Z$-action on $X$ generated by~$S$. 
	
	Now, Lemma \ref{tom} can be applied to the $\Z$-action on $X$ given by the iterates of $S$ in the role of the action of $\Gamma$, the classical F\o lner \sq\ $F_n=[0,n]$ in $\Z$, and $g$ (viewed as a mapping on $X$) in the role of $T$. The lemma yields that for $\mu$-almost every $x\in X$, any $\eps>0$ and $n$ sufficiently large, we have 
	\begin{equation*}
		|\{x,S(x),\dots, S^n(x)\}\ \triangle\ g(\{x,S(x),\dots, S^n(x)\})|\le \eps(n+1).	
	\end{equation*}
	By \eqref{TD}, the above means that
	\begin{equation*}
		|[0,n]^\prec(x)\ \triangle\ g([0,n]^\prec)(x)|\le \eps(n+1),
	\end{equation*}
	where $\prec\ =\varphi(x)$. Because the action of $G$ on $X$ is free, we can skip $x$, and get
	\begin{equation*}
		|[0,n]^\prec\ \triangle\ g([0,n]^\prec)|\le \eps(n+1).
	\end{equation*}
	We have shown that for $\nu$-almost every $\prec\ \in\tilde\O$, any $g\in G$ and every $\eps>0$, the order intervals $[0,n]^\prec$ are eventually $(g,\eps)$-\inv. This ends the proof.
\end{proof}

\section{Conditional entropy with respect to a multiorder}
This section contains the formula for the conditional entropy of a measure-preserving $G$-action which has a multiorder as a measure-theoretic factor.

\subsection{Preparatory lemmas}
Throughout this subsection we assume that $(\tilde\O,\nu,G)$ is a multiorder on a countable amenable group $G$. The F\o lner property of the multiorder, although it holds, is not used yet. Nor we assume zero entropy of $\nu$. The following two lemmas will be very useful. 

\begin{lem}\label{nono}
	Let $(F_m)_{m\in\N}$ be a fixed F\o lner \sq\ in $G$. We also fix some $n\in\N$ and $\eps>0$.
	Then, for sufficiently large $m\in\N$ there exists a subset $\tilde\O'_m\subset\tilde\O$ with $\nu(\tilde\O'_m)>1-\eps$, such that for every $\prec\ \in\tilde\O'_m$ we have
	\begin{equation*}
		\frac{|[F_m, F_m+n]^\prec|}{|F_m|}<1+\eps \text{ \ and \ }\frac{|[F_m-n, F_m]^\prec|}{|F_m|}<1+\eps
	\end{equation*}
	(see \eqref{TD} for the meaning of $[F_m,F_m+n]^\prec$ and $[F_m-n,F_m]^\prec$).
\end{lem}

\begin{proof}
	
	Firstly, observe that $|[F_m-n,F_m]^\prec|=|[F_m,F_m+n]^\prec|$ for all $\prec\, \in\tilde{\O}$, so we can focus on satisfying the first inequality only. Secondly, note that if the lemma holds for all ergodic measures $\nu$ then, by a standard decomposition argument, it holds for all \im s $\nu$ as well. Thus, we can assume ergodicity of the measure $\nu$. 
	
	Let $\mathcal K$ denote the family of all subsets of $G$ of cardinality $n\!+\!1$ containing the unit and let $\delta=\frac{\varepsilon}{2n+2}$. We have the following disjoint union:
	\begin{equation*}
		\tilde\O=\bigsqcup_{K\in\K}\tilde\O_K,
	\end{equation*}
	where $\tilde\O_K=\{\prec\ \in\tilde\O:[0,n]^\prec\, = K\}$. 
	Because the union is countable, there exists a finite subset $\K'\subset\K$ such that, if we denote $\tilde\O_{\K'}=\bigsqcup_{K\in\K'}\tilde\O_K$, then
	\begin{equation*}
		\nu(\tilde\O_{\K'})=\sum_{K\in\K'}\nu(\tilde\O_K)>1-\delta.
	\end{equation*}
	Let
	\begin{equation*}
		L=\bigcup_{K\in\K'}K
	\end{equation*}
	(clearly, $L$ is a finite subset of $G$).
	For sufficiently large $m$, the set $F_m$ is $(L,\frac\delta{|L|})$-\inv. Then the $L$-core\footnote{For two finite sets $F$ and $K$, the \emph{$K$-core} of $F$ (usually denoted by $F_K$) is defined as the set $\{g\in F: Kg\subset F\}$.} of $F_m$, which we denote by $(F_m)_L$, has cardinality at least $(1-\delta)|F_m|$.\footnote{We are using the elementary fact that if a set $F$ is $(K,\eps)$-\inv\ then the $K$-core $F_K$ has cardinality at least $(1-|K|\eps)|F|$, see e.g. \cite[Lemma 2.6]{DHZ}.} By the mean ergodic theorem, 
	for sufficiently large $m$ there exists a set $\tilde\O'_m\subset\tilde\O$ with $\nu(\tilde\O'_m)>1-\eps$, such that for all $\prec\ \in\tilde\O'_m$ we have
	\begin{equation*}
		|\{g\in F_m: g(\prec)\in \tilde\O_{\K'}\}|>(1-\delta)|F_m|.
	\end{equation*}
	Combining this with the estimate of the cardinality of the $L$-core $(F_m)_L$, we obtain, for $m$ large enough,
	\begin{equation*}
		|\{g\in (F_m)_L: g(\prec)\in \tilde\O_{\K'}\}|>(1-2\delta)|F_m|.
	\end{equation*}
	Observe that $g(\prec)\in \tilde\O_{\K'}\iff [0,n]^{g(\prec)}\in\K'$. By \eqref{16} and \eqref{17}, for $k$ satisfying $g=k^\prec$, we can write
	\begin{equation*}
		[0,n]^{g(\prec)}=[k,k+n]^{\prec}\cdot g^{-1}=[g,g+n]^\prec\cdot g^{-1}.
	\end{equation*}
	Then $[g,g+n]^\prec\,=Kg$, where $K\in\K'$, implying $[g,g+n]^\prec\subset Lg$. If also $g\in(F_m)_L$ then $Lg\subset F_m$ and hence $[g,g+n]^\prec\subset F_m$. Summarizing, we have shown that, if we denote by $F'_m$ the set $\{g\in (F_m)_L: g(\prec)\in \tilde\O_{\K'}\}$ (which is a subset of $F_m$ of cardinality strictly larger than $(1-2\delta)|F_m|$), then $[F_m',F_m'+n]^\prec\subset F_m$. Obviously,
	\begin{equation*}
		[F_m,F_m+n]^\prec\setminus F_m\subset[F_m,F_m+n]^\prec\setminus[F'_m,F'_m+n]^\prec\subset[(F_m\setminus F'_m), (F_m\setminus F'_m)+n]^\prec,
	\end{equation*}
	and hence we conclude that
	\begin{equation*}
		|[F_m,F_m+n]^\prec\setminus F_m|\le (n+1)|F_m\setminus F'_m|< (n+1)2\delta|F_m|=\eps|F_m|,
	\end{equation*}	
	which is exactly what we needed to show. 
\end{proof}

\begin{lem}\label{trud}
	Let $J:\tilde\O\to\mathbb R$ be a bounded measurable function. Then, for each $j\in\N$ we have
	\begin{equation*}
		\int J(j^\prec(\prec))\,\mathrm d\nu=\int J(\prec)\,\mathrm d\nu\ \footnote{Whenever we write $\int \cdots \mathrm d\nu$, we mean $\int\cdots \mathrm d\nu(\prec)$, i.e.\ we never use integration with respect to $\nu$, where the variable is denoted by a symbol different from $\prec$.}
	\end{equation*}
	(here $j^\prec(\prec)$ stands for the image $g(\prec)$ of $\prec$ by the element $g=j^\prec$ in the action defined by~\eqref{ac}).
\end{lem}

\begin{proof}
	This is a direct consequence of Theorem~\ref{motooe}, more preciscely of the facts that the successor map $\tilde S$ on $\tilde \O$ preserves the measure $\nu$, and $j^{\prec}(\prec)={\tilde S}^j(\prec)$ for all $\prec\,\in\tilde\O$.
\end{proof}

\subsection{Entropy of a multiordered system}\label{tu}

\begin{defn}
	By a \emph{multiordered dynamical system} (or, more precisely, a \emph{multiordered action} of $G$ on $X$), denoted by $(X,\mu,G,\varphi)$, we will mean a measure-preserving $G$-action $(X,\mu,G)$ with a fixed measure-theoretic factor map\break$\varphi:(X,\mu,G)\to(\tilde\O,\nu,G)$ to a multiorder $\tilde\O$ equipped with an \im\ $\nu$ (provided such a factor map exists).
\end{defn}
Any dynamical system with an action of a countable amenable group can be turned into a multiordered one by joining it (for example via the product joining), with a multiorder. In such case, the factor map $\varphi$ is by default the projection onto the second coordinate. Moreover, by Theorem~\ref{exi}, we can always choose to use a~multiorder of entropy zero, in which case the joining maintains the entropy of the original system.

\begin{defn}\label{past}Let $(X,\mu,G,\varphi)$ be a multiordered \ds. Fix a finite measurable partition $\P$ of $X$. For a subset $D\subset G$ we will denote
	\begin{equation*}
		\P^D=\bigvee_{g\in D}g^{-1}(\P).
	\end{equation*}
	The sigma-algebras $\P^-_\prec\,=\P^{(-\infty,-1]^\prec}$ and $\P^+_\prec\,=\P^{[1,\infty)^\prec}$ are called respectively the \emph{past} and the \emph{future} of $\P$ with respect to $\prec\ \in\tilde\O$.
\end{defn}

Before we continue, we establish some basic definitions concerning entropy in actions of countable amenable groups. Given a measure-preserving system $(X,\mu,\Sigma,G)$ and a finite measurable partition $\P$ of $X$, by the entropy of $\P$ we will mean
\begin{equation*}
	H(\mu,\P)=-\sum_{P\in\P}\mu(P)\log{\mu(P)}.
\end{equation*}
If $\Theta$ is a $G$-invariant sub-sigma-algebra of $\Sigma$, then the conditional entropy of $\P$ with respect to $\Theta$ equals
\begin{equation*}
	H(\mu,\P|\Theta)=\inf_{\Q}\bigl(H(\mu,\P\vee\Q)-H(\mu,\Q)\bigr),
\end{equation*} 
where $\Q$ ranges over all finite partitions of $X$ measurable with respect to~$\Theta$.
The dynamical entropy of the process generated by $\P$ is defined by the formula
\begin{equation*}
	h(\mu,G,\P)=\lim\limits_{n\to\infty}\frac{1}{|F_n|}H(\mu,\P^{F_n}),
\end{equation*}
where $(F_n)_{n\in\N}$ is a F{\o}lner sequence in $G$ (the definition does not depend on a F{\o}lner sequence).
The conditional entropy of the process generated by $\P$ with respect to~$\Theta$ equals
\begin{equation*}
	h(\mu,G,\P|\Theta)=\lim\limits_{n\to\infty}\frac{1}{|F_n|}H(\mu,\P^{F_n}|\Theta).
\end{equation*}
The Kolmogorov--Sinai entropy of a dynamical system $h(\mu,G)$ is by definition the supremum of $h(\mu,G,\P)$ over all finite measurable partitions $\P$ of~$X$.
\begin{thm}\label{newentropy} Let $(X,\mu,G,\varphi)$ be a multiordered \ds\ and let $\P$ be a finite measurable partition of $X$. Let $\{\mu_{\prec}:\,\prec\ \in\tilde\O\}$ denote the disintegration of~$\mu$ with respect to $\nu=\varphi(\mu)$. Let $\Sigma_{\tilde\O}$ denote the invariant sub-sigma-algebra on $X$ obtained by lifting the Borel sets in $\tilde\O$ against the factor map $\varphi$.
	Then
	\begin{equation*}
		h(\mu,G,\P|\Sigma_{\tilde\O})=\int H(\mu_{\prec},\P|\P_{\prec}^-)\,\mathrm d\nu = \int H(\mu_{\prec},\P|\P_{\prec}^+)\,\mathrm d\nu. 
	\end{equation*}
\end{thm}

\begin{rem}\label{amr}
	For an IRO we have a similarly looking formula (see \cite{AMR}):
	\begin{equation*}
		h(\mu,G,\P) = \int H(\mu,\P|\P_{\prec}^-)\,\mathrm d\nu.
	\end{equation*}
	The proof of the above formula is much shorter than that of Theorem \ref{newentropy}. In particular, the proof of Theorem \ref{newentropy} relies (seemingly inevitably) on the F\o lner property of the multiorder, which is absent in \cite{AMR}. The simplicity of the proof in \cite{AMR} seems to rely on the fact that the measure $\mu$ under the integral does not depend on~$\prec$. 
\end{rem}

\begin{proof}[Proof of Theorem \ref{newentropy}]
	We define an auxiliary entropy notion, as follows
	\begin{equation}\label{hbar}
		\bar h(\mu,\P,\varphi)=\lim_{n\to\infty}\frac1{n+1}\int H(\mu_{\prec},\P^{[0,n]^{\prec}})\,\mathrm d\nu.
	\end{equation}
	For every $n\in\N$ we have the following equality:
	\begin{multline*}
		H(\mu_{\prec},\P^{[0,n]^{\prec}})=\\
		H(\mu_{\prec},\P^{[0]^{\prec}})+H(\mu^{\prec},\P^{[1]^{\prec}}|\P^{[0]^{\prec}})+
		H(\mu_{\prec},\P^{[2]^{\prec}}|\P^{[0,1]^{\prec}})+\dots\\+
		H(\mu_{\prec},\P^{[n]^{\prec}}|\P^{[0,n-1]^{\prec}\}}).
	\end{multline*}
	Let us consider just the $j$th term of the above sum (keeping in mind that it will be eventually integrated with respect to $\nu$):
	\begin{equation*}
		H(\mu_{\prec},\P^{[j]^{\prec}}|\P^{[0,j-1]^{\prec}})
	\end{equation*}
	(for $j=0$ the conditioning partition disappears). Since the disintegration is that of an \im, it is equivariant,\footnote{Equivarance follows immediately from invariance of $\mu$ and uniqueness of the disintegration.} i.e.\ for any $g\in G$, it satisfies, for each measurable set $A\subset X$,
	\begin{equation}\label{eqv}
		\mu_{\prec}(A)=\mu_{g(\prec)}(g(A)).
	\end{equation}
	This implies that
	\begin{equation*}
		H(\mu_{\prec},\P^{[j]^{\prec}}|\P^{[0,j-1]^{\prec}})=
		H(\mu_{g(\prec)},g(\P^{[j]^{\prec}})|g(\P^{[0,j-1]^{\prec}})).
	\end{equation*}
	Note that, for any $D\subset G$ we have
	\begin{equation}\label{prze}
		g(\P^D)=\bigvee_{h\in D}(gh^{-1})(\P)=\bigvee_{h\in D}(hg^{-1})^{-1}(\P)=\P^{Dg^{-1}}.
	\end{equation}
	In particular, we obtain that
	\begin{equation*}
		H(\mu_{\prec},\P^{[j]^{\prec}}|\P^{[0,j-1]^{\prec}})=
		H(\mu_{g(\prec)},\P^{[j]^{\prec}g^{-1}}|\P^{[0,j-1]^{\prec}g^{-1}}).
	\end{equation*}
	
	So far $g\in G$ was arbitrary. Now let $g=j^{\prec}$. For each $i\in\{0,1,\dots,n\}$, by \eqref{16} and~\eqref{17}, we have
	\begin{equation*}
		i^{\prec}\cdot g^{-1}=(i-k)^{g(\prec)},
	\end{equation*} 
	where $k$ is such that $g=k^{\prec}$. But since $g = j^\prec$, we have $k=j$ and we conclude that $i^{\prec}\cdot g^{-1}=(i-j)^{g(\prec)}=(i-j)^{j^\prec(\prec)}$. So
	\begin{equation}\label{theabove}
		H(\mu_{\prec},\P^{[j]^{\prec}}|\P^{[0,j-1]^{\prec}})=
		H(\mu_{j^{\prec}(\prec)},\P^{[0]^{j^{\prec}(\prec)}}|\P^{[-j,-1]^{j^{\prec}(\prec)}}).
	\end{equation}
	If we define a measurable function $J:\tilde\O\to[0,\infty)$ by 
	\begin{equation*}
		J(\prec)=H(\mu_\prec,\P^{[0]^\prec}|\P^{[-j,-1]^\prec})
	\end{equation*}
	then \eqref{theabove} can be rewritten as
	\begin{equation*}
		H(\mu_{\prec},\P^{[j]^{\prec}}|\P^{[0,j-1]^{\prec}})=J(j^{\prec}(\prec)).
	\end{equation*}
	Since the function $J$ is bounded (by $\log{|\P|}$), we can use Lemma \ref{trud} and get 
	\begin{equation*}
		\int J(j^{\prec}(\prec))\,\mathrm d\nu = \int J(\prec)\,\mathrm d\nu,
	\end{equation*}
	which means that
	\begin{equation}\label{long}
		\int H(\mu_{\prec},\P^{[j]^{\prec}}|\P^{[0,j-1]^{\prec}})\,\mathrm d\nu=
		\int H(\mu_{\prec},\P^{[0]^{\prec}}|\P^{[-j,-1]^{\prec}})\,\mathrm d\nu.
	\end{equation}
	Now we can go back to the formula \eqref{hbar} defining $\bar h(\mu,\P,\varphi)$ and substitute the integrals according to \eqref{long}. Since $0^{\prec}=e$ and hence $\P^{[0]^{\prec}}=\P$ (for any $\prec\ \in\tilde\O$), we get
	\begin{equation*}
		\bar h(\mu,\P,\varphi)=\int\lim_{n\to\infty} \frac1{n+1}\sum_{j=0}^n
		H(\mu_{\prec},\P|\P^{[-j,-1]^{\prec}})\,\mathrm d\nu
	\end{equation*}
	(we have used the Lebesgue theorem to exchange the limit with the integral). Note that for each $\prec$, the \sq\ $H(\mu_{\prec},\P|\P^{[-j,-1]^{\prec}})$ indexed by $j$ converges nonincreasingly to $H(\mu_{\prec},\P|\P_{\prec}^-)$ (we use continuity of entropy with respect to a refining \sq\ of partitions, see e.g. \cite[Lemma 1.7.11]{D}). By monotonicity, the \sq\ of the arithmetic averages appearing in the last integral has the same limit. We have proved that
	\begin{equation}\label{submain}
		\bar h(\mu,\P,\varphi)=\int H(\mu_{\prec},\P|\P_{\prec}^-)\,\mathrm d\nu.
	\end{equation}
	The proof of the dual formula
	\begin{equation}
		\bar h(\mu,\P,\varphi)=\int H(\mu_{\prec},\P|\P_{\prec}^+)\,\mathrm d\nu
	\end{equation}
	is identical.
	\smallskip
	
	To complete the proof we need to show that $\bar h(\mu,\P,\varphi)=h(\mu,G,\P|\Sigma_{\tilde\O})$.
	We will prove the two respective inequalities separately.
	
	Fix a finite set $K\subset G$ and some $\eps>0$. By the F\o lner property of the multiorder (Theorem \ref{TM}), there exists a set $\tilde\O'\subset\tilde\O$ with $\nu(\tilde\O')>1-\eps$, and $n_0\in\N$ such that, for any $\prec\ \in\tilde\O'$ and $n\ge n_0$, the order-interval $[0,n]^{\prec}$ is $(K,\frac\eps{|K|})$-\inv. Then the $K$-core $[0,n]^{\prec}_K$ of such an interval occupies the fraction of at least $1-\eps$ in that interval. Hence, for any $\prec\ \in\tilde\O'$, we have:
	\begin{equation}\label{tru}
		H(\mu_{\prec},\P^{[0,n]^{\prec}})\le H(\mu_{\prec},\P^{[0,n]^{\prec}_K})+\eps(n+1)\log|\P|.
	\end{equation}
	Observe that, for any $\prec\ \in\tilde\O$, an element $h\in [0,n]^{\prec}_K$ belongs to $K^{-1}g$ for some $g\in [0,n]^{\prec}$ if and only if $g\in Kh$. Since $Kh\subset [0,n]^{\prec}$, there are exactly $|K|$ such elements $g$. This means that the family $\{K^{-1}g:g\in [0,n]^{\prec}\}$ is a so-called $|K|$-cover of the core $[0,n]^{\prec}_K$. 
	Thus, we can apply the Shearer's inequality (see e.g.\ \cite[Section~2]{DF}) and obtain the following:
	\begin{equation}\label{truq}
		H(\mu_{\prec},\P^{[0,n]^{\prec}_{K}})\le \frac1{|K|}\sum_{g\in [0,n]^{\prec}}H(\mu_{\prec},\P^{K^{-1}g}).
	\end{equation}
	By integrating, we get
	\begin{multline*}
		\int H(\mu_{\prec},\P^{[0,n]^{\prec}}) \,\mathrm d\nu=
		\int_{\tilde\O'} H(\mu_{\prec},\P^{[0,n]^{\prec}}) \,\mathrm d\nu+ \int_{\tilde\O\setminus\tilde\O'} H(\mu_{\prec},\P^{[0,n]^{\prec}}) \,\mathrm d\nu\le\\
		\int_{\tilde\O'} H(\mu_{\prec},\P^{[0,n]^{\prec}}) \,\mathrm d\nu+ \eps(n+1)\log|\P|\overset{\eqref{tru}}\le\\
		\int_{\tilde\O'} H(\mu_{\prec},\P^{[0,n]^{\prec}_K}) \,\mathrm d\nu+ 2\eps(n+1)\log|\P|\le\\
		\int H(\mu_{\prec},\P^{[0,n]^{\prec}_K}) \,\mathrm d\nu+ 2\eps(n+1)\log|\P|\overset{\eqref{truq}}\le\\
		\frac1{|K|}\int\sum_{g\in [0,n]^{\prec}} H(\mu_{\prec},\P^{K^{-1}g})\,\mathrm d\nu+2\eps(n+1)\log|\P|\overset{\eqref{eqv}}=\\
		\frac1{|K|}\int\sum_{g\in [0,n]^{\prec}} H(\mu_{g(\prec)},g(\P^{K^{-1}g}))\,\mathrm d\nu+2\eps(n+1)\log|\P|=\cdots.
	\end{multline*}
	
	By \eqref{prze}, we have $g(\P^{K^{-1}g}) = \P^{K^{-1}}$. We can also write $g\in[0,n]^\prec$ as $i^\prec$ with $i\in[0,n]$. Thus, we can continue as follows
	\begin{equation*}
		\dots=\frac1{|K|}\sum_{i\in [0,n]}\int H(\mu_{i^{\prec}(\prec)},\P^{K^{-1}})\,\mathrm d\nu+2\eps(n+1)\log|\P|.
	\end{equation*}
	By Lemma \ref{trud}, for every $i\in[0,n]$, we have 
	\begin{equation*}
		\int H(\mu_{i^{\prec}(\prec)},\P^{K^{-1}})\,\mathrm d\nu=\int H(\mu_{\prec},\P^{K^{-1}})\,\mathrm d\nu.
	\end{equation*}
	Because this term does not depend on $i$, the sumation over $i\in[0,n]$ becomes multiplication by $(n+1)$. We can now divide both sides by $(n+1)$ and finish our inequality:
	\begin{multline*}
		\frac1{n+1}\int H(\mu_{\prec},\P^{[0,n]^{\prec}})\,\mathrm d\nu\le\\
		\frac1{n+1}\frac1{|K|}(n+1)\int H(\mu_{\prec},\P^{K^{-1}})\,\mathrm d\nu+2\eps\log|\P|=\\
		\frac1{|K|}H(\mu,\P^{K^{-1}}|\Sigma_{\tilde\O})+2\eps\log|\P|,
	\end{multline*}
	where the last equality is just the standard formula for the conditional entropy (given a sub-sigma-algebra) via disintegration of the measure (see e.g.\ \cite[formula~(1.5.4)]{D}).
	
	Passing to the limit over $n$ on the left hand side and then passing to the limit over a F\o lner \sq\ $(K_m)_{m\in\N}$ (in place of $K^{-1}$) together with $\eps_m\to 0$ (in place of $\eps$) on the right hand side, we conclude that
	\begin{equation*}
		\bar h(\mu,\P,\varphi)\le h(\mu,G,\P|\Sigma_{\tilde\O}).
	\end{equation*}
	\smallskip
	
	We proceed to proving the converse inequality. Fix some $n\in \N$ and $\eps>0$. For any $\prec\ \in\tilde\O$, the family  $\{[g,g+n]^{\prec}:g\in [F_m-n,F_m]^\prec\}$ is easily seen to be an $(n\!+\!1)$-cover of $F_m$, hence the Shearer's inequality can be used again, as follows:
	\begin{equation*}
		H(\mu_{\prec},\P^{F_m})\le\frac1{n+1}\sum_{g\in [F_m-n,F_m]^\prec}H(\mu_{\prec},\P^{[g,g+n]^{\prec}}).
	\end{equation*}
	By Lemma \ref{nono}, for large enough $m$ and any $\prec\ \in\tilde\O'_m$, where $\tilde\O'_m\subset\tilde\O$ satisfies $\nu(\tilde\O'_m)>1-\eps$, we have
	\begin{equation*}
		|[F_m-n,F_m]^\prec\setminus F_m|\le\eps|F_m|,
	\end{equation*}
	and thus 
	\begin{equation*}
		H(\mu_{\prec},\P^{F_m})\le\frac1{n+1}\sum_{g\in F_m}H(\mu_{\prec},\P^{[g,g+n]^{\prec}})+\eps|F_m|\log|\P|.
	\end{equation*}
	By \eqref{eqv} and \eqref{prze}, we have $H(\mu_{\prec},\P^{[g,g+n\,]^{\prec}})=H(\mu_{g(\prec)},\P^{[g,g+n\,]^{\prec}\cdot g^{-1}})$ for any $g\in G$. By~\eqref{16} and~\eqref{17}, we also have $[g,g+n]^{\prec}\cdot g^{-1}=[0,n]^{g(\prec)}$. We obtain~that
	\begin{equation}\label{om1}
		H(\mu_{\prec},\P^{F_m})\le\frac1{n+1}\sum_{g\in F_m}H(\mu_{g(\prec)},\P^{[0,n]^{g(\prec)}})+\eps|F_m|\log|\P|.
	\end{equation}
	
	We now integrate with respect to $\nu$ and get
	\begin{multline*}\label{om}
		H(\mu,\P^{F_m}|\Sigma_{\tilde\O})=
		\int H(\mu_\prec,\P^{F_m})\,\mathrm d\nu=\\
		\int_{\tilde\O'_m} H(\mu_\prec,\P^{F_m})\,\mathrm d\nu
		+\int_{\tilde\O\setminus\tilde\O'_m} H(\mu_\prec,\P^{F_m})\,\mathrm d\nu\overset{\eqref{om1}}\le\\
		\frac1{n+1}\sum_{g\in F_m}\int H(\mu_{g(\prec)},\P^{[0,n]^{g(\prec)}})\,\mathrm d\nu+\eps|F_m|\log|\P|+\eps|F_m|\log|\P|\overset{\text{inv\!.\,of\,}\nu}\le\\
		\frac1{n+1}\sum_{g\in F_m}\int H(\mu_{\prec},\P^{[0,n]^{\prec}})\,\mathrm d\nu+2\eps|F_m|\log|\P|=\\
		\frac{|F_m|}{n+1}\int H(\mu_{\prec},\P^{[0,n]^{\prec}})\,\mathrm d\nu+2\eps|F_m|\log|\P|,
	\end{multline*}
	where the last equality holds because the summation of terms not depending on $g$ becomes multiplication by $|F_m|$. 
	
	We can now divide all terms by $|F_m|$ and pass to the limit over $m$. Then, on the extreme left, we obtain $h(\mu,G,\P|\Sigma_{\tilde\O})$. After dividing by $|F_m|$, the right hand side no longer depends on $m$, so we have proved that
	\begin{equation*}
		h(\mu,G,\P|\Sigma_{\tilde\O})\le \frac1{n+1}\int H(\mu_{\prec},\P^{[0,n]^{\prec}})\,\mathrm d\nu+2\eps\log|\P|.
	\end{equation*}
	By passing to the limit as $n\to\infty$ we can replace the right hand side by\break${\bar h(\mu,\P,\varphi)+2\eps\log|\P|}$. Since $\eps$ is arbitrary, the last term  can be skipped and we obtain $h(\mu,G,\P|\Sigma_{\tilde\O})\le\bar{h}(\mu,\P,\varphi)$. This ends the proof.
\end{proof}
\begin{rem}
	As pointed out by Alexandre Danilenko \cite{D0}, Theorem~\ref{newentropy} can be alternatively proved basing on a more general approach developed in \cite{D1}.
\end{rem}
\begin{cor}\label{produkt}
	Given a $G$-action $(X,\mu,G)$ and an arbitrary multiorder $(\tilde\O,\nu,G)$, consider the product system $(X\times\tilde\O, \mu\times\nu, G)$ with the product action  $g(x,\prec)=(g(x),g(\prec))$. This product is a multiordered system with the projection on the second coordinate in the role of the factor map $\varphi$.
	In this case, the disintegration of $\mu\times\nu$ with respect to $\nu$ is constant, i.e.\
	$\mu_\prec\, = \mu$ for all $\prec\ \in\tilde\O$. Moreover, by independence, for any partition $\P$ of $X$ (which by lifting can be considered a partition of $X\times\tilde{\O}$) we have $h(\mu\times\nu,G,\P|\Sigma_{\tilde\O})=h(\mu,G,\P)$. Hence, the formula in Theorem \ref{newentropy} takes on the form
	\begin{equation*}
		h(\mu,G,\P)= \int H(\mu,\P|\P_{\prec}^-)\,\mathrm d\nu= \int H(\mu,\P|\P_{\prec}^+)\,\mathrm d\nu,
	\end{equation*}
	i.e.\ we recover the formula from \cite{AMR} (see Remark \ref{amr}).
\end{cor}

\begin{cor}
	In case of a multiordered system $(X,\mu,G,\varphi)$ such that the associated multiorder $(\tilde\O,\nu,G)$ has entropy zero, we have $h(\mu,G,\P|\Sigma_{\tilde\O})=h(\mu,G,\P)$, and thus Theorem \ref{newentropy} provides a formula for the unconditional entropy:
	\begin{equation*}
		h(\mu,G,\P) = \int H(\mu_\prec,\P|\P_{\prec}^-)\,\mathrm d\nu= \int H(\mu_\prec,\P|\P_{\prec}^+)\,\mathrm d\nu.
	\end{equation*}
\end{cor}

\section{Preservation of conditional entropy under the orbit equivalence determined by a multiorder}\label{siedem}
In this section we continue to study a multiordered dynamical system $(X,\mu,G,\varphi)$. We prove the equality between the conditional (with respect to the multiorder) entropy of the $G$-action and the analogous conditional entropy of the $\Z$-action given by the iterates of the successor maps $S$ and $\tilde{S}$ defined by the formulae \eqref{full} and~\eqref{SforO}. A well oriented reader will note that this equality follows from a theorem of Rudolph and Weiss \cite[Theorem~2.6]{RW} but, on the one hand, our proof is completely different and much shorter (even when counting the necessary background), on the other hand, as we will show in the next section, our theorem implies that of Rudolph--Weiss (with slightly changed assumptions).

\begin{thm}\label{eq_entropies}
	Let $(X,\mu,G,\varphi)$ be a multiordered dynamical system and let $S$ denote the successor map defined by the formula \eqref{full}. Then, for every finite, measurable partition $\P$ of $X$ we have,
	\begin{equation}
		h(\mu,G,\P|\Sigma_{\tilde{\O}})=h(\mu,S,\P|\Sigma_{\tilde{\O}}),
	\end{equation}
	where $h(\mu,G,\P|\Sigma_{\tilde{\O}})$ is the conditional (with respect to $\Sigma_{\tilde\O}$) entropy of the process $(X,\mu,\P,G)$ generated by $\P$ under the action of $G$ and $h(\mu,S,\P|\Sigma_{\tilde{\O}})$ is the analogous conditional entropy of the process $(X,\mu,\P,S)$ generated by $\P$ under the action of $\Z$ given by the iterates of $S$.
\end{thm}
\begin{proof}
	For every $n\in \Z$, we denote $\P^{(-\infty,-n]^S}=\bigvee_{k\ge n}S^k(\P)$. Firstly, we show that for $\nu$-almost every $\prec\,\in\tilde\O$, we have 
	\begin{equation}\label{Sprec}
		\P^{(-\infty,-n]^S}|_{\varphi^{-1}(\prec)}=\P^{(-\infty,-n]^{\prec}}|_{\varphi^{-1}(\prec)}.
	\end{equation}
	Let $\prec\,\in\tilde{\O}$ and $k\in\Z$ be fixed. By the formula \eqref{TD}, for every measurable set $A\subset X$, we have
	\begin{gather*}
		S^k(A)\cap \varphi^{-1}(\prec)=S^k\bigl(\bigcup_{\prec'\in\tilde{\O}}A\cap \varphi^{-1}(\prec')\bigr)\cap\varphi^{-1}(\prec)=\\ \Bigl(\bigcup_{\prec'\in\tilde{\O}} S^k(A\cap \varphi^{-1}(\prec'))\Bigr)\cap \varphi^{-1}(\prec)=\Bigl(\bigcup_{\prec'\in\tilde{\O}} k^{\prec'}(A\cap \varphi^{-1}(\prec'))\Bigr)\cap \varphi^{-1}(\prec)=\\\Bigl(\bigcup_{\prec'\in\tilde{\O}} \bigl(k^{\prec'}(A)\cap \varphi^{-1}(k^{\prec'}(\prec'))\bigr)\Bigr)\cap \varphi^{-1}(\prec).
	\end{gather*}
	The only $\prec'\in\tilde{\O}$ for which the corresponding item of the above union has nonempty intersection with $\varphi^{-1}(\prec)$ is the one for which $k^{\prec'}(\prec')=\,\prec$. Since $k^{\prec'}(\prec')={\tilde S}^k(\prec')$, we have $k^{\prec'}(\prec')=\,\prec$ if and only if $\prec'=\tilde{S}^{-k}(\prec)=(-k)^{\prec}(\prec)$ (see formula~\eqref{TD} applied to $\tilde{\O}$ as its own extension). Thus, by the first part of formula \eqref{16} (with $g=(-k)^{\prec}$), we obtain 
	\begin{equation*}
		k^{\prec'}=k^{(-k)^{\prec}(\prec)}=(k-k)^{\prec}\cdot \bigl((-k)^{\prec}\bigr)^{-1}=\bigl((-k)^{\prec}\bigr)^{-1}.
	\end{equation*}
	Hence,
	\begin{equation}\label{Ska}
		S^k(A)\cap \varphi^{-1}(\prec)=\bigl((-k)^{\prec}\bigr)^{-1}(A)\cap\varphi^{-1}(\prec).
	\end{equation}
	Therefore, $S^k(\P)|_{\varphi^{-1}(\prec)}=((-k)^{\prec})^{-1}(\P)|_{\varphi^{-1}(\prec)}$ and for every $n\in \Z$,
	\begin{multline*}
		\P^{(-\infty,-n]^S}|_{\varphi^{-1}(\prec)}=\bigvee_{k\ge n}S^k(\P)|_{\varphi^{-1}(\prec)}=\\\bigvee_{k\ge n}\bigl((-k)^{\prec}\bigr)^{-1}(\P)|_{\varphi^{-1}(\prec)}=\P^{(-\infty,-n]^\prec}|_{\varphi^{-1}(\prec)}.
	\end{multline*}
	In particular, $\P^{(-\infty,-1]^S}\bigr|_{\varphi^{-1}(\prec)}=\P_{\prec}^-\bigr|_{\varphi^{-1}(\prec)}$.
	
	Consequently, by Theorem \ref{newentropy} and by the disintegration formula for the conditional entropy (see e.g.\ \cite[page~255]{P}, the conditional version passes by the same proof), we have
	\begin{multline}\label{cor_eq}
		h(\mu,G,\P|\Sigma_{\tilde O})=\int H(\mu_{\prec},\P|\P_{\prec}^-)\,\mathrm d\nu= \int H\bigl(\mu_{\prec},\P|\P^{(-\infty,-1]^S}\bigr)\,\mathrm d\nu=\\H(\mu,\P|\P^{(-\infty,-1]^S}\vee \Sigma_{\tilde\O})=h(\mu,S,\P|\Sigma_{\tilde O}).
	\end{multline}
	This ends the proof.
\end{proof}
\begin{rem}
	As pointed out by Alexandre Danilenko \cite{D0}, Theorem~\ref{eq_entropies} can be alternatively proved basing on a more general approach developed in \cite{D1}.
\end{rem}

\begin{cor}\label{surp}
	Suppose $\varphi:X\to\tilde\O$ is a measure-theoretic factor map from a~measure-preserving $G$-action $(X,\mu,G)$ to a multiorder $(\tilde\O,\nu,G)$ and let $(X,\mu,S)$ be the $\Z$-action orbit equivalent to $(X,\mu,G)$ as described in Theorem~\ref{motooe}. Recall that the theorem establishes also that the $\Z$-action $(X,\mu,S)$ factors to $(\tilde{\O},\nu,\tilde{S})$ via the same map $\varphi$. It is a well-known phenomenon that orbit equivalent systems may have different entropies. However, in the above situation, it follows from Theorem~\ref{eq_entropies} that it is the multiorder factor which is responsible for the entire difference of entropies:
	\begin{equation}\label{diff_ent}
		h(\mu,G)-h(\mu,S)=h(\nu,G)-h(\nu,\tilde{S})
	\end{equation}
	(assuming that $h(\mu,S)<\infty$ and hence also $h(\nu,\tilde{S})<\infty$).  
\end{cor}
\begin{exam}\label{below}
	In general, there is no inequality between $h(\mu,G)$ and $h(\mu,S)$. Indeed,
	let $(X_1,\mu_1,T_1)$ and $(X_2,\mu_2,T_2)$ be two ergodic $\Z$-actions on atomless probability spaces, such that $h(\mu_1,T_1)>h(\mu_2,T_2)$ (alternatively $h(\mu_1,T_1)<h(\mu_2,T_2)$). The first action will be viewed as a $G$-action $(X_1,\mu_1,G)$, where $G=\Z$. By the theorem of Dye these actions are orbit equivalent. By the comments at the beginning of Section \ref{trzy}, we can assume that $X_1=X_2$, $\mu_1=\mu_2$ and the orbit-equivalence is established by the identity map. By Remark~\ref{pooetomo}, the first action factors onto a multiorder $(\tilde{\O},\nu,G)$ such that the successor map $S$ associated to that multiorder coincides with $T_2$. Then $h(\mu_1,G)>h(\mu_1,S)$ (alternatively $h(\mu_1,G)<h(\mu_1,S)$). 
\end{exam}

\section{Proof of the Rudolph-Weiss Theorem via multiorders}
This section is devoted to showing that our Theorem \ref{eq_entropies} not only follows but also implies the Rudolph--Weiss Theorem. In fact, it implies a slightly less general version, but the loss of generality is marginal in comparison to the gain of simplicity.

We begin by quoting the original theorem and stating our version.
\begin{thm}\label{RWa}\cite[Theorem~2.6]{RW}
	Let $(X,\mu,G)$ and $(X,\mu,\Gamma)$ be free actions of two countable amenable groups $G$ and $\Gamma$, with the same orbits. Let $(Y,\nu,G)$ be a factor of $(X,\mu,G)$ and let $\Sigma_Y$ denote the associated $G$-invariant sub-sigma-algebra on $X$. Assume that the orbit change from the action of $G$ to the action of $\Gamma$ is measurable with respect to $\Sigma_Y$, i.e.\ that for all $\gamma\in \Gamma$ and $g\in G$ the sets $\{x\in X: g(x)=\gamma(x)\}$ belong~to~$\Sigma_Y$. Then $\Sigma_Y$ is $\Gamma$-invariant and
	\begin{equation}
		h(\mu, G\bigl|\Sigma_Y)=h(\mu,\Gamma\bigl|\Sigma_Y).
	\end{equation}
\end{thm}

\begin{thm}\label{RWb}
	Let $(X,\mu,G)$ and $(X,\mu,\Gamma)$ be actions of two countable amenable groups $G$ and $\Gamma$, with the same orbits. Let $\Sigma_Y$ be a sub-sigma-algebra on $X$ invariant under both actions, and such that both these actions on the corresponding factor space $(Y,\nu)$ are free. Then
	\begin{equation}\label{eq_entropies_RW}
		h(\mu,G\bigl|\Sigma_Y)=h(\mu,\Gamma|\Sigma_Y).
	\end{equation}
\end{thm}

We will now discuss the differences between these two formulations. There are two changes in the assumptions. 
\begin{enumerate}[(1)]
	\item In Theorem \ref{RWa} we assume freeness of the actions of $G$ and $\Gamma$ on $(X,\mu)$, while in Theorem \ref{RWb} we assume the same about the actions on the common factor $(Y,\nu)$ (which is clearly a stronger assumption).  
	\item We replace the $\Sigma_Y$-measurability of the change of orbits by the ``double invariance'' of~$\Sigma_Y$ (which is seemingly a weaker assumption). 
\end{enumerate} 
The next lemma shows that although the change (2) seems to be in favor of Theorem~\ref{RWb} (see part a)), in view of the change (1) it is actually not (see part b)). So Theorem \ref{RWb} is slightly, but strictly less general.\footnote{Consider two identical free actions $(X,\mu,G)=(X,\mu,\Gamma)$, where $G=\Gamma$, and the trivial factor $\Sigma_Y=\{X,\varnothing\}$. As easily verified, this example satisfies the assumptions of Theorem \ref{RWa} but not of Theorem~\ref{RWb}.}

\begin{lem}\label{equiv_RW}
	Let $(X,\mu,G)$ and $(X,\mu,\Gamma)$ be two actions of countable amenable groups, which have the same orbits. The following implications are true:
	\begin{enumerate}[a)]
		\item If $(Y,\nu,G)$ is a factor of $(X,\mu,G)$ (i.e.\ $\Sigma_Y$ is a $G$-invariant sub-sigma-algebra on $X$), and the orbit change from the action of $G$ to the action of $\Gamma$ is $\Sigma_Y$-measurable, then $\Sigma_Y$ is $\Gamma$-invariant (i.e.\ $(Y,\nu,\Gamma)$ is a factor of $(X,\mu,\Gamma)$).
		\item If $\Sigma_Y$ is a sub-sigma-algebra on $X$ which is both $G$-\inv\ and $\Gamma$-invariant, and the action of at least one of the groups on the corresponding factor space $(Y,\nu)$ is free, then the orbit change from $(X,\mu,G)$ to $(X,\mu,\Gamma)$ is $\Sigma_Y$-measurable.
	\end{enumerate}
\end{lem}
\begin{proof}\hfill
	\begin{enumerate}[a)]
		\item If $(Y,\nu,G)$ is a factor of $(X,\mu,G)$, then for every $A\in\Sigma_Y$ and $\gamma\in\Gamma$ we have
		\begin{equation*}
			\gamma^{-1}(A)=\{x\in X: \gamma(x)\in A\}=\bigcup_{g\in G}\bigl(\{x\in X: \gamma(x)=g(x)\}\cap g^{-1}(A)\bigr).
		\end{equation*}
		Since the orbit change is assumed $\Sigma_Y$-measurable,  $\{x\in X: \gamma(x)=g(x)\}\in\Sigma_Y$. Next, $g^{-1}(A)\in \Sigma_Y$ by $G$-invariance of $\Sigma_Y$. Hence $\gamma^{-1}(A)\in\Sigma_Y$.
		\item Assume that $\Sigma_Y$ is a sub-sigma-algebra on $X$ which is both $G$-\inv\ and $\Gamma$-invariant. Let $\pi:X\to Y$ be the corresponding factor map (note that the factor map does not depend on the choice of the acting group). Assume that one of the groups, say $G$, acts freely on $(Y,\nu)$. Fix some $\gamma\in\Gamma$ and $y\in Y$. Let $x$ and $x'$ both belong to $\pi^{-1}(y)$. Let $g,g'\in G$ be such that $g(x)=\gamma(x)$ and $g'(x')=\gamma(x')$. We need to show that $g=g'$ (this will imply that the orbit change is constant on the fibers of $\pi$, i.e.\ $\Sigma_Y$-measurable). We have
		\begin{multline*}
			\hspace{10pt}g'(y) = g'(\pi(x'))=\pi(g'(x'))=\pi(\gamma(x'))=\gamma(\pi(x'))=\gamma(y)=\\\gamma(\pi(x))=\pi(\gamma(x))=\pi(g(x))=g(\pi(x))=g(y).
		\end{multline*}
		Since the $G$-action on $(Y,\nu)$ is free, we conclude that $g=g'$, as needed.
	\end{enumerate}
\end{proof}	
The following example shows that the assumption of freeness of the factor is essential.
\begin{exam}
	Let $(X,\mu,T)$ be any free $\Z$-action and let $(Y,\nu,\phi)$ be given by $Y=\{0,1\}$, $\nu(\{0\})=\nu(\{1\})=\tfrac12$ and $\phi(y)=1-y$. On $Y\times X$ consider the skew product:
	\begin{equation*}
		T_{\phi}(y,x)=\begin{cases}
			(\phi(y),x)& \text{ for $y=0$}\\
			(\phi(y),T(x))& \text{ for $y=1$}.
		\end{cases}
	\end{equation*}
	On the same product space consider also the $\Z_2\times\Z$ action given by
	\begin{equation*}
		(m,n)(y,x)=(\phi^m(y),T^n(x)),\ m\in\Z_2,\ n\in\Z.
	\end{equation*}
	These actions are free and have the same orbits. On $\{0\}\times X$ we have $T_{\phi}(y,x)=(1,0)(y,x)$, while on $\{1\}\times X$ we have $T_{\phi}(y,x)=(1,1)(y,x)$. The trivial sigma-algebra $\Sigma_{\mathsf{triv}}$ on $Y\times X$ is clearly invariant under both actions, but the sets $\{0\}\times X$ and $\{1\}\times X$ (each of product measure $\tfrac{1}{2}$) are not measurable with respect to $\Sigma_{\mathsf{triv}}$.
\end{exam}
The next observation will be used in the proof of Theorem \ref{RWb}. 
\begin{lem}\label{ext_oe}
	Let $(X,\mu,G)$ be a measure-preserving action of a countable amenable group $G$. Let $(Y,\nu,G)$ be a factor of $(X,\mu,G)$, such that the action of $G$ on $(Y,\nu)$ is free. Let $\pi:X\to Y$ be the corresponding factor map. Let $(Y,\nu,\Gamma)$ be an action of a countable amenable group $\Gamma$ on $(Y,\nu)$, which has the same orbits as $(Y,\nu,G)$. Then the action $(Y,\nu,\Gamma)$ has a unique extension
	$(X,\mu,\Gamma)$ via the map $\pi$ with the same orbits as $(X,\mu,G)$. The extension is defined by the formula
	\begin{equation}\label{eq:1}
		\gamma(x)=g(x),
	\end{equation}
	where $g$ is the unique element of $G$ such that $g(\pi(x))=\gamma(\pi(x))$.
\end{lem}

\begin{proof} Obviously, the formula \eqref{eq:1} defines an action of $\Gamma$ on $(X,\mu)$ which extends $(Y,\nu,\Gamma)$ via $\pi$, and has the same orbits as $(X,\mu,G)$. Now, let $(X,\mu,\Gamma)$ denote any extension of $(Y,\nu,\Gamma)$ via $\pi$, with the same orbits as $(X,\mu,G)$. By Lemma \ref{equiv_RW} (and since $G$ acts freely on $(Y,\nu)$), the orbit change from $(X,\mu,\Gamma)$ to $(X,\mu,G)$ is $\Sigma_Y$-measurable, which means precisely what we desire: for any $\gamma\in\Gamma$ and $x\in X$, we have $\gamma(x)=g(x)$, where $g\in G$ is the unique element such that $\gamma(\pi(x))=g(\pi(x))$.
\end{proof}

\begin{proof}[Proof of Theorem \ref{RWb}]
	Assume that the actions of $G$ and $\Gamma$ on $(Y,\nu)$ are free. If $\nu$~is not ergodic, then its ergodic decomposition does not depend on the action. Since the actions of $G$ and $\Gamma$ are free with respect to $\nu$, they are clearly free with respect to almost every ergodic component of $\nu$. Thus, it suffices to consider $\nu$ ergodic. 
	
	By the theorem of Ornstein and Weiss (\cite[Theorem~6]{OW}) and the Dye Theorem~\cite{Dy}, there exists a $\Z$-action $(Y,\nu,T)$ of entropy zero, with the same orbits as $(Y,\nu,G)$ (and thus as $(Y,\nu,\Gamma)$). 
	
	Since the action $(Y,\nu,G)$ is free, by Remark \ref{pooetomo}, there exists a multiorder $(\tilde{\O}_G,\nu_G,G)$ being a factor of $(Y,\nu,G)$ via a factor-map $\varphi_G$, such that $T$ coincides with the successor map $S_G$ associated with the multiorder $\tilde{\O}_G$. Similarily, there exists a multiorder $(\tilde{\O}_{\Gamma},\nu_\Gamma,\Gamma)$ being a factor of $(Y,\nu,\Gamma)$ via a factor-map $\varphi_{\Gamma}$, such that $T$ coincides with the successor map $S_{\Gamma}$ associated with the multiorder $\tilde{\O}_{\Gamma}$. These two multiorders are factors of $(X,\mu,G)$ and $(X,\mu,\Gamma)$ via $\varphi_G\circ\pi$ and $\varphi_\Gamma\circ\pi$, respectively, and they induce on $(X,\mu)$ two (\textit{a priori} different) successor maps, which we will denote by $\bar S_G$ and $\bar S_\Gamma$. The corresponding $\Z$-actions on $(X,\mu)$ have the same orbits as $(X,\mu,G)$ and $(X,\mu,\Gamma)$. By Lemma \ref{ext_oe}, (applied to $(Y,\nu,T)$ rather than $(Y,\nu,\Gamma)$) we know that there exists a unique extension $\bar T$ of $T$ onto $(X,\mu)$ which has the same orbits as $(X,\mu,G)$ and $(X,\mu,\Gamma)$. In order to show that $\bar S_G=\bar S_\Gamma$ it suffices to show that they both equal $\bar T$ (we will do that only for $\bar S_G$, the proof for $\bar S_\Gamma$ is analogous).
	
	Recall that, by definition, $\bar{S}_G(x)=1^{\prec}(x)$, where $\prec\,=(\varphi_G\circ\pi)(x)$. Letting $y=\pi(x)$ we can write $\prec\,=\varphi_G(y)$. On the other hand, by \eqref{eq:1} applied to $1\in\Z$ (in the role of $\gamma$) we get $\bar T(x)=g(x)$, where $g\in G$ is unique such that $g(y)=T(y)$. Since $T=S_G$, we have $g(y)=S_G(y)=1^{\prec}(y)$. Since the action of $G$ on $(Y,\nu)$ is free, we have $g=1^{\prec}$, hence $\bar T(x)=1^{\prec}(x)=\bar{S}_G(x)$. As explained earlier, we have just proved that $\bar S_G=\bar S_\Gamma$.
	
	We are in the position to use our Theorem \ref{eq_entropies} with respect to the four multiordered systems $(Y,\nu,G,\varphi_G)$, $(Y,\nu,\Gamma,\varphi_{\Gamma})$, $(X,\mu,G,\varphi_G\circ\pi)$ and $(X,\mu,\Gamma,\varphi_{\Gamma}\circ\pi)$.
	We obtain the following equalities
	\begin{gather*}
		h(\nu,G\bigl|\Sigma_{\tilde{\O}_G})=h(\nu,T\bigl|\Sigma_{\tilde{\O}_G})=0,\\
		h(\nu,\Gamma\bigl|\Sigma_{\tilde{\O}_{\Gamma}})=h(\nu,T\bigl|\Sigma_{\tilde{\O}_{\Gamma}})=0,\\
		h(\mu,G|\Sigma_{\tilde{\O}_G})=h(\mu,\bar T\bigl|\Sigma_{\tilde{O}_G}),\\
		h(\mu,\Gamma|\Sigma_{\tilde{\O}_{\Gamma}})=h(\mu,\bar T\bigl|\Sigma_{\tilde{O}_{\Gamma}}).
	\end{gather*}
	Hence we also have 
	\begin{gather*}
		h(\mu,G\bigl|\Sigma_Y)=h(\mu,G|\Sigma_{\tilde{\O}_G})-h(\nu,G|\Sigma_{\tilde{\O}_G})=h(\mu,G|\Sigma_{\tilde{\O}_G})-0=h(\mu,\bar T\bigl|\Sigma_{\tilde{O}_G}),\\
		h(\mu,\Gamma\bigl|\Sigma_Y)=h(\mu,\Gamma|\Sigma_{\tilde{\O}_{\Gamma}})-h(\nu,\Gamma|\Sigma_{\tilde{\O}_{\Gamma}})=h(\mu,\Gamma|\Sigma_{\tilde{\O}_{\Gamma}})-0=h(\mu,\bar T\bigl|\Sigma_{\tilde{O}_{\Gamma}}).
	\end{gather*}
	Since $\Sigma_{\tilde{\O}_G}$ and $\Sigma_{\tilde{\O}_\Gamma}$ are $T$-invariant sub-sigma-algebras on $Y$, the action of $T$ on the corresponding factors have entropy zero. Thus they are also zero-entropy factors of $(X,\mu,\bar T)$, which implies that
	\begin{equation*}
		h(\mu,\bar T\bigl|\Sigma_{\tilde{O}_G})=h(\mu,\bar T\bigl|\Sigma_{\tilde{O}_{\Gamma}})=h(\mu,\bar T).
	\end{equation*}
	Therefore, we get the desired equality
	\begin{equation*}
		h(\mu,G\bigl|\Sigma_Y)=h(\mu,\Gamma\bigl|\Sigma_Y) \ \ (=h(\mu,\bar T)),
	\end{equation*}
	and the proof is complete.
\end{proof}
\begin{rem}
	Alexandre Danilenko \cite{D0} indicates that both the Rudolph-Weiss Theorem and our Theorem~\ref{RWb} are covered by \cite[Theorem~0.3]{D1}.
\end{rem}
\section{Pinsker sigma-algebra via multiorder}
As an application of Theorem \ref{eq_entropies} we provide a characterization of the Pinsker sigma-algebra relative to the multiorder $\tilde{\O}$, in terms of the $\Z$-action given by the iterates of the successor map $S$. Recall that the Pinsker sigma-algebra $\Pi_G(X\bigr|\Theta)$ of a dynamical system $(X,\mu,G)$, with respect to a $G$-invariant sub-sigma-algebra~$\Theta$ consists of all measurable sets $A\subset X$ such that $h(\mu,G,\{A,A^c\}\bigr|\Theta)=0$. If $\Sigma_Y$ is a $G$-\inv\ sub-sigma-algebra on $X$ and $(Y,\nu,G)$ is the corresponding factor of $(X,\mu,G)$, then $\Pi_G(Y\bigr|\Theta)=\Pi_G(X\bigr|\Theta)\cap\Sigma_Y$. In case $(Y,\nu,G)$ is generated by a~finite measurable partition $\P$ of $X$ and $\Theta$ is trivial, we will write $\Pi_G(\P)$ instead of $\Pi_G(Y\bigr|\Theta)$.
\begin{thm}\label{pinsker_mult}
	Let $(X,\mu,G,\varphi)$ be a multiordered dynamical system with an action of a countable amenable group $G$. Let $\tilde\O=\varphi(X)$ denote the multiorder factor. The Pinsker sigma-algebra $\Pi_G(X\bigr|{\Sigma_{\tilde\O}})$ of the system $(X,\mu,G)$, relative to the multiorder~$\tilde\O$ is equal to the Pinsker sigma-algebra $\Pi_S(X\bigr|{\Sigma_{\tilde\O}})$ of the $\Z$-action $(X,\mu,S)$, relative to the multiorder $\tilde\O$, where $S$ is the successor map on $X$ associated with the multiorder $\tilde\O$.
\end{thm}
\begin{proof}
	Let $\P$ be a finite measurable partition of $X$. Then $\P$ is measurable with respect to $\Pi_G(X\bigr|{\Sigma_{\tilde\O}})$ if and only if $h(\mu,G,\P|\Sigma_{\tilde{\O}})=0$, if and only if ${h(\mu,S,\P|\Sigma_{\tilde{\O}})=0}$ (by Theorem~\ref{eq_entropies}), if and only if $\P$ is measurable with respect to $\Pi_S(X\bigr|{\Sigma_{\tilde\O}})$.
\end{proof}
\begin{cor}
	Suppose that $(X,\mu,G,\varphi)$ is a multiordered dynamical system such that the multiorder factor $(\tilde\O,\nu,G)$ has entropy zero. In such case, the (unconditional) Pinsker sigma-algebra $\Pi_G(X)$ of this action equals $\Pi_S(X|\Sigma_{\tilde{O}})$.
\end{cor}
\begin{rem}
	Since, in general, there is no connection between entropy of a multiorder under the action of $G$ and under the $\Z$-action generated by the iterates of~$\tilde S$, we cannot claim that $\Pi_G(X)=\Pi_{S}(X)$. This equality holds, however, if $\tilde O$ has ``double entropy zero'' (see Corollary~\ref{35}).
\end{rem}

Despite the above remark, Theorem \ref{pinsker_mult} allows to characterize the (unconditional) Pinsker factor of an arbitrary measure-preserving action of a countable amenable group $G$ by a formula which resembles the Rokhlin--Sinai formula for $\Z$-actions $\Pi_{T}(\P)=\bigcap_{n\ge 1}\P^{(-\infty,-n]}$. The result sheds a new light on the Pinsker factor even in the classical case of $G=\Z$. It turns out that, in addition to the well-known fact that $\Pi_{T}(\P)$ can be computed using either the remote past or the remote future (which coresponds to two natural orders on $\Z$), it can be expressed using an appropriately understood remote past (or future) with respect to any (non-standard) multiorder on $\Z$, for instance such as given in Example~\ref{Z-ord} in Appendix~B.
\begin{thm}\label{R-s}
	Let $(X,\Sigma,\mu,G)$ be a measure-theoretic dynamical system with an action of a countable amenable group $G$. Let $\P$ be a finite measurable partition of~$X$. Fix a multiorder $(\tilde{\O},\Sigma_{\tilde{\O}},\nu,G)$ on $G$. A set $A$ belongs to the Pinsker sigma-algebra $\Pi_{G}(\P)$ of the $G$-process generated by $\P$ if and only if:
	\begin{equation}\label{R-s_eq}
		A\in \bigcap_{n\ge 1} \P^{(-\infty,-n]^{\prec}},
	\end{equation}
	for $\nu$-almost every $\prec\,\in\tilde{\O}$. Equivalently, for $\nu$-almost every $\prec\,\in\tilde{\O}$,
	\begin{equation}\label{R-s_eq2}
		\Pi_{G}(\P)= \bigcap_{g\in G}\bigcap_{n\ge 1} \P^{(-\infty,-n]^{g(\prec)}}.
	\end{equation}	
\end{thm}
\begin{proof}
	Let $(Y,\Sigma_Y,\mu_Y,G)$ denote the symbolic factor of~$(X,\Sigma,\mu,G)$ generated by~$\P$ and let $(\tilde{\O},\Sigma_{\tilde{\O}},\nu,G)$ be an arbitrary multiorder on $G$. Consider the product dynamical system ($Y\times \tilde{\O},\Sigma_Y\otimes \Sigma_{\tilde{\O}},\mu_Y\times\nu,G)$. As remarked in Corollary \ref{produkt}, this is a multiordered $G$-action with the projection on the second coordinate in the role of the factor map $\varphi$. It is clear that the Pinsker sigma-algebra $\Pi_{G}(Y\times\tilde{\O}\bigr|\Sigma_{\tilde{\O}})$ contains the product sigma-algebra
	$\Pi_{G}(\P)\otimes\Sigma_{\tilde\O}$.
	
	In view of Theorem~\ref{pinsker_mult}, we now have 
	\begin{equation*}
		\Pi_{G}(Y\times \tilde{\O}\bigr|\Sigma_{\tilde{\O}})=\Pi_S(Y\times\tilde{\O}\bigr|{\Sigma_{\tilde\O}}),
	\end{equation*} 
	where $S$ is the successor map on $Y\times\tilde\O$ associated to the multiorder factor via the projection~$\varphi$.
	We can view $\P$ and $\Sigma_{\tilde{\O}}$, as a partition and sigma-algebra in $Y\times\tilde{\O}$, respectively. Notice that $\P\vee\Sigma_{\tilde{\O}}$ generates $\Sigma_Y\otimes\Sigma_{\tilde{\O}}$ not only under the action of $G$ but also under the action of $S$. Indeed, consider two different pairs $(y,\prec)$ and $(y',\prec')$ in $Y\times\tilde{\O}$. Either $\prec\,\neq\,\prec'$, in which case the pairs are clearly separated by $\P\vee\Sigma_{\tilde{\O}}$, or $\prec\,=\,\prec'$ and $y\neq y'$. In the latter case the iterates $S^n$ act on these pairs by the same elements of $G$, moreover, as $n$ ranges over $\Z$, these elements exhaust the whole group. This implies that, for large enough $n$, the partitions $\P^{[-n,n]^{S}}$ separate the considered pairs. Now, by a~well-known formula for $\Z$-actions (see e.g.~\cite{Z}) the Pinsker sigma-algebra $\Pi_S(Y\times\tilde{\O}\bigr|\Sigma_{\tilde{\O}})$ is equal to the intersection $\bigcap_{n\ge 1}\bigl(\P^{(-\infty,-n]^S}\vee \Sigma_{\tilde{O}}\bigr)$. We conclude that
	\begin{equation}\label{PScos}
		\Pi_{G}(\P)\otimes\Sigma_{\tilde\O}\preccurlyeq\bigcap_{n\ge 1}\bigl(\P^{(-\infty,-n]^S}\vee \Sigma_{\tilde{O}}\bigr).
	\end{equation}
	
	Consider a set $A\in \Pi_{G}(\P)$. For each $n\ge 1$, we have $A\times \tilde{\O}\in{ \P^{(-\infty,-n]^S}\vee\Sigma_{\tilde{O}}}$, which implies that for $\nu$-almost every $\prec\,\in\tilde{\O}$, the intersection $(A\times \tilde{\O})\cap\varphi^{-1}(\prec)$ belongs to $\P^{(-\infty,-n]^S}\bigr|_{\varphi^{-1}(\prec)}=\P^{(-\infty,-n]^{\prec}}\bigr|_{\varphi^{-1}(\prec)}$ (see formula \eqref{Sprec}). Notice that $\P^{(-\infty,-n]^{\prec}}$ depends only on the $G$-action on $X$. Hence, we conclude that for $\nu$-almost every $\prec\,\in \tilde{\O}$, $A$~is measurable with respect to $\P^{(-\infty,-n]^{\prec}}$. Since this is true for every $n\ge 1$, it follows that $A\in \bigcap_{n\ge 1}\P^{(-\infty,-n]^{\prec}}$ for $\nu$-almost every $\prec\,\in\tilde{\O}$.
	
	Conversely, consider a set $A\subset Y$ measurable with respect to $\bigcap_{n\ge 1}\P^{(-\infty,-n]^{\prec}}$ for $\nu$-almost every $\prec\,\in\tilde{\O}$. Let  $C=A\times\tilde{\O}$. The intersection $C\cap\varphi^{-1}(\prec)$ is measurable with respect to $\bigcap_{n\ge 1}\P^{(-\infty,-n]^{\prec}}\bigr|_{\varphi^{-1}(\prec)}=\bigcap_{n\ge 1}\P^{(-\infty,-n]^{S}}\bigr|_{\varphi^{-1}(\prec)}$ for $\nu$-almost every $\prec\,\in\tilde{\O}$, which, together with measurability of $C$ with respect to $\Sigma_Y\times\Sigma_{\tilde\O}$, implies that $C$ is measurable with respect to $\bigl(\bigcap_{n\ge 1}\P^{(-\infty,-n]^S}\bigr)\vee\Sigma_{\tilde{\O}}$ (see e.g.\ comments following Lemma~1.2.2 in \cite{D}).
	It is clear that $\bigl(\bigcap_{n\ge 1}\P^{(-\infty,-n]^S}\bigr)\vee \Sigma_{\tilde{O}}$ is refined by (but in general is not equal to) $\bigcap_{n\ge 1}\bigl(\P^{(-\infty,-n]^S}\vee \Sigma_{\tilde{O}}\bigr)$. So $C$ is measurable with respect to $\Pi_S(Y\times\tilde{\O}\bigr|\Sigma_{\tilde{\O}})=\Pi_G(Y\times\tilde{\O}\bigr|\Sigma_{\tilde{\O}})$. This means that if $\Q=\{C,C^c\}=\{A\times\tilde{\O},A^c\times{\tilde{\O}}\}$, then $h(\mu_Y\times\nu,G,\Q\bigr|\Sigma_{\tilde{\O}})=0$. Since $G$ acts independently on each axis, this implies that $h(\mu_Y,G,\R)=0$, where $\R$ is the partition $\{A,A^c\}$ of~$Y$. As a consequence $A\in\Pi_G(\P)$. The proof of \eqref{R-s_eq} is now complete. The only nontrivial part of \eqref{R-s_eq2} is the inclusion of sigma-algebras
	\begin{equation}\label{R-S_eq100}
		\mathcal{A}_{\prec}:=\bigcap_{g\in G}\bigcap_{n\ge 1} \P^{(-\infty,-n]^{g(\prec)}}\preccurlyeq \Pi_G(\P),
	\end{equation}
	for $\nu$-almost every $\prec$. It suffices to prove this for an ergodic measure $\nu$. For any $\prec\,\in\tilde{\O}$ the sigma-algebra $\mathcal{A}_{\prec}$ is $G$-invariant. Using ergodicity of $\nu$, it can be shown that $\mathcal{A}_{\prec}$ is the same sigma-algebra, henceforth denoted by $\mathcal{A}$, for a set of $\prec$ of full measure $\nu$. Since $\mathcal{A}\preccurlyeq \bigcap_{n\ge 1} \P^{(-\infty,-n]^{\prec}}$ for $\nu$-almost every $\prec\,\in\tilde{\O}$, by \eqref{R-s_eq}, we have $\mathcal{A}\preccurlyeq \Pi_{G}(\P)$, which implies that $\mathcal{A}_{\prec}\preccurlyeq \Pi_{G}(\P)$ for $\nu$-almost every $\prec\,\in\tilde{\O}$.
	
\end{proof}
\begin{rem}
	As pointed out by Alexandre Danilenko \cite{D0}, Theorem~\ref{R-s} can be alternatively proved basing on a more general approach developed in \cite{D1}.
\end{rem}
\begin{appendix}
\section{Ordered tiling systems}\label{til} 
\addtocontents{toc}{\protect\setcounter{tocdepth}{1}}

In this section of the Appendix we summarize some facts concerning tilings and systems of tilings of amenable groups introduced and studied in \cite{DHZ} and we propose a new notion of an ordered tiling system. 

\subsection{General tilings}
Let $G$ be a countable group. 
\begin{defn}A \emph{tiling} $\T$ of $G$ is a partition of $G$ into (countably many) finite sets (called \emph{tiles}), i.e.\ we have
	\begin{equation*}
		G=\bigsqcup_{T\in\T}T \text{\ \ (disjoint union)}.
	\end{equation*}
\end{defn}

\begin{defn}\label{prop} A tiling $\T$ is \emph{proper} if there exists a finite collection $\S$ of finite sets $S\in\S$ (not necessarily different) each containing the unit $e$, called the \emph{shapes} of $\T$, such that for every $T\in\T$ there exists a shape $S\in\S$ satisfying $T=Sc$ for some $c\in G$ (in fact, we then have $c\in T$). 
\end{defn}
From now on all tilings we will be dealing with are assumed to be proper.  
For a~(proper) tiling $\T$ we will always fix one collection of shapes $\S$ and one representation $T\mapsto(S,c)$, where $S\in\S, c\in G$ are such that $T=Sc$. (We remark that, in general, there may be more than one such representation, even when $\S$ is fixed and contains no pairs of equal sets.) Once such a representation is fixed, we will call $S$ and $c$ the \emph{shape} and \emph{center} of $T$, respectively. Given $S\in\S$, we will denote by $C_S(\T)$ the set of centers of the tiles having the shape $S$, while $C(\T)=\bigsqcup_{S\in\S}C_S(\T)$ will be used to denote the set of centers of all tiles.

\subsection{Dynamical tilings}

Let $\T$ be a tiling with the collection of shapes $\S$. Denote by $\rm V$ the finite alphabet consisting of symbols assigned bijectively to the shapes of~$\T$ plus one additional symbol:
\begin{equation}\label{alfa}
	\rm V = \{``S": S\in\S\}\cup\{``0"\}.
\end{equation} 
Then $\T$ can be identified with the symbolic element, denoted by the same letter $\T\in{\rm V}^G$, defined as follows: 
\begin{equation*}
	\T(g) = \begin{cases}
		``\!S" \text{ for some }S\in\S, &\text{if } g\in C_S(\T),\\ 
		``0", & \text{otherwise}.
	\end{cases}
\end{equation*}
\begin{defn}
	Let $\rm V$ be an alphabet of the form \eqref{alfa} for some finite collection $\S$ of finite sets $S$. Let $\TT\subset {\rm V}^G$ be a subshift such that each element $\T\in\TT$ represents a  tiling with the collection of shapes contained in $\S$. Then we call $\TT$ a \emph{dynamical tiling} and $\S$ \emph{the collection of shapes} of $\TT$. 
\end{defn}
It is elementary to see that the orbit-closure (under the shift-action of $G$) of any  tiling $\T$ is a dynamical tiling.

\subsection{Systems of tilings and tiling systems}
In the sequel we will be using a~very special \tl\ joining of dynamical tilings. By
a~\emph{\tl\ joining} of a~\sq\ of \ds s $(X_k,G)$, $k\in\N$, (denoted by $\bigvee_{k\in\N}X_k$)\footnote{We remark that the symbol $\bigvee_{k\in\N}X_k$ refers to many possible topological joinings.} we mean any closed subset of the Cartesian product $\prod_{k\in\N}X_k$ which has full projections onto the coordinates $X_k$, $k\in\N$, and is \inv\ under the product action given by $g(x_1,x_2,\dots)=(g(x_1),g(x_2),\dots)$. 

\begin{defn}\label{tili}
	Consider a \sq\ of dynamical tilings $(\TT_k)_{k\in \N}$. By a \emph{system of tilings} (generated by the dynamical tilings $\TT_k$) we mean any \tl\ joining $\TTT=\bigvee_{k\in\N}\TT_k$. 
\end{defn}

The elements of $\TTT$ have the form of \sq s of tilings $\boldsymbol\T=(\T_k)_{k\in\N}$, where $\T_k\in\TT_k$ for each $k$.

\begin{defn}\label{tili1}
	Let $\TTT=\bigvee_{k\in\N}\TT_k$ be a system of tilings and let $\S_k$ denote the collection of shapes of $\TT_k$. The system of tilings is:
	\begin{itemize}
		\item \emph{congruent}, if for each $\boldsymbol{\T}=(\T_k)_{k\in\N}\in\TTT$ and each 	
		$k\in\N$, every tile of $\T_{k+1}$ is a union of some tiles of $\T_k$.
		\item \emph{deterministic}, if it is congruent and for every $k\ge 1$ and any ${S'\in\S_{k+1}}$, there exist sets $C_{S}(S')\subset S'$ indexed by $S\in\S_k$, such that
		\begin{equation*}
			S'=\bigsqcup_{S\in\S_k}\ \bigsqcup_{c\in C_S(S')}Sc 	
		\end{equation*}
		and for each $\boldsymbol{\T}=(\T_i)_{i\in\N}\in\TTT$, whenever $S'c'$ is
		a tile of $\T_{k+1}$, then 
		\begin{equation*}
			S'c'=\bigsqcup_{S\in\S_k}\ \bigsqcup_{c\in C_{S}(S')}Scc'
		\end{equation*}
		is the partition of $S'c'$ by the tiles of $\T_k$. 
	\end{itemize}
\end{defn}
\begin{rem}\label{remd} In a deterministic system of tilings, for each $\boldsymbol{\T}=(\T_k)_{k\in\N}\in\TTT$, each tiling $\T_{k'}$ \emph{determines} all the tilings $\mathcal T_k$ with $k<k'$, and the assignment $\T_{k'}\mapsto\T_k$ is a \tl\ factor map from $\T_{k'}$ onto $\T_k$. In such a case, the joining $\TTT$ is in fact an inverse limit
	\begin{equation*}
		\TTT=\overset\longleftarrow{\lim_{k\to\infty}}\mathsf{T}_k.
	\end{equation*}
	The inverse limit will not change if the sequence $(\mathsf{T}_k)_{k\in\N}$ is replaced by a subsequence. We will call such a replacement \emph{speeding up the tiling system}.
\end{rem}

The next two definitions apply to amenable groups only.
\begin{defn}\label{fst}
	A system of tilings $\TTT$ is \emph{F\o lner}, if the union of the collections of shapes $\bigcup_{k\in\N}\S_k$ (arranged in a \sq) is a F\o lner \sq\ in $G$.
\end{defn}
\begin{defn}
	Any F\o lner, deterministic and minimal\footnote{A \tl\ \ds\ is \emph{minimal} if it contains no proper closed invariant subsets.} system of tilings $\TTT$ will be called simply a \emph{tiling system} (not to be confused with much less organized system of tilings).
\end{defn}

The following theorem will play a crucial role in our considerations:	

\begin{thm}(\cite[Theorem 5.2]{DHZ})\label{dhz} Every countable amenable group admits a tiling system with \tl\ entropy~zero.
\end{thm}

\begin{rem}\label{mini}
	\cite[Theorem 5.2]{DHZ} asserts the existence of a deterministic F\o lner system of tilings $\TTT$ with \tl\ entropy~zero (not necessarily minimal). However, since $\TTT$ obviously contains a minimal subsystem, there also exists a minimal system of tilings with all the above properties. 
\end{rem}

\subsection{Ordered tiling systems}
Every tiling system can be equipped with a partial order, as described below. Recall, that for every $k\ge2$, every shape $S\in\S_k$ is partitioned into subtiles of order $k-1$:
\begin{equation*}
	S=\bigsqcup_{S'\in\S_{k-1}}\ \bigsqcup_{c'\in C_{S'}(S)}S'c'.
\end{equation*}
\begin{rem}For the above to make sense also for $k=1$ we agree that $\S_0 = \{\{e\}\}$, i.e.\ we introduce the tiling of order $0$, $\T_0$, as the tiling whose all tiles are singletons. Notice that this tiling is a fixpoint of the shift action, hence constitutes a one-element dynamical tiling $\TT_0=\{\T_0\}$.
\end{rem}

Let $C(S)$ be the set of all centers of the subtiles of $S$, i.e.\ $C(S)=\bigsqcup_{S'\in\S_{k-1}}C_{S'}(S)$. In $C(S)$ we fix some ordering, as follows:
\begin{equation*}
	C(S) = \{c_1^S,c_2^S,\dots,c_{l(S)}^S\},
\end{equation*}
where $l(S) = |C(S)|$. Then the partition of $S$ into subtiles of order $k-1$ also becomes ordered:
\begin{equation}\label{decs}
	S=\bigsqcup_{i=1}^{l(S)}S'_ic_i^S,
\end{equation}
where $S'_i\in\S_{k-1}$ and $c_i^S \in C_{S'_i}(S)$, for each $i=1,2,\dots,l(S)$. 

\begin{defn} By an \emph{ordered tiling system} we will mean a tiling system $\TTT$ with the ordering of subtiles established for each shape $S\in\S_k$ ($k\in\N$).
\end{defn}

Observe that the notion of subtiles and their ordering applies not only to the shapes $S\in\S_k$ but also to any tile of any $\T_k\in\TT_k$. Indeed, every such tile, say $T$, has the form $Sc$ where $S\in\S_k$ then its subtiles are naturally ordered as follows:
\begin{equation}\label{dect}
	T = \bigsqcup_{i=1}^{l(S)}S'_ic_i^Sc.
\end{equation}
Notice that for any $k'>k$ and $S'\in\S_{k'}$ the above ordering induces, in a natural way (lexicographically), an order on the subtiles of $S'$ with shapes in $\S_k$.
We will use this property in the next section, when speeding up an ordered tiling system.
\section{Tiling-based multiorder}\label{til-ord}
This section of the Appendix is devoted to proving a strengthened version of Theorem~\ref{exi} by invoking the machinery of ordered tiling systems.

\begin{defn}
	Let $\TTT$ be a tiling system and let $\bT=(\T_k)_{k\in\N}\in\TTT$. We say that 
	$\bT$ is \emph{in general position} if the central tiles of $\T_k$ cover $G$, that is, if
	\begin{equation*}
		\bigcup_{k\in\N} T^e_k = G,
	\end{equation*}
	where $T_k^e$ denotes the central tile of $\T_k$, i.e.\ the tile containing the unit $e$.
\end{defn}
Observe that in any congruent system of tilings, the central tiles always form an increasing (with respect to inclusion) \sq, so the above union is increasing. We will soon show that the subset of $\TTT$ consisting of those $\bT$ which are in general position is both topologically large, i.e.\ residual, and large in the sense of measure, i.e.\ has measure $1$ for every \im\ on $\TTT$ (in fact, we will show this for an even smaller set).

Now let $\TTT$ be an ordered tiling system and let $\bT=(\T_k)_{k\in\N}\in\TTT$ be in general position. Then $\bT$ determines a linear ordering of $G$, by the following rule:

Let $a\neq b\in G$. There exists $k\ge 1$ such that $a,b$ belong to a common tile of $\T_k$ (indeed, eventually they belong to a common central tile of some $\T_k$). Let $k(a,b)$ be the smallest such index $k$, and let $T=Sc$ ($S\in\S_{k(a,b)}$) be the tile of $\T_{k(a,b)}$ which contains $a,b$. By the definition of $k(a,b)$, $a$ and $b$ belong to different subtiles of $T$. We say that $a\prec_{_{\bT}} b$ (or $a\succ_{_{\!\!\bT}} b$) if $a$ belongs to the subtile of $T$ which precedes (follows) the subtile containing $b$ in the ordering \eqref{dect} of the subtiles of $T$.

It is not hard to see that the orders $\prec_{_{\bT}}$ will not change if we speed up the ordered tiling system $\TTT$ (and appropriately compose the ordering of subtiles).

Observe that for any $a,b\in G$ with $a\prec_{_{\bT}} b$ there are at most finitely many elements $g\in G$ such that $a\prec_{_{\bT}} g\prec_{_{\bT}} b$ (indeed, all such elements must belong to the common tile $T$ of $\T_{k(a,b)}$). This implies that $(G,\prec_{_{\bT}})$ is order-isomorphic to either $(\Z,<)$ or $(\N,<)$, or $(-\N,<)$, with the two latter cases occurring if the central tiles $T^e_{k-1}$ have eventually (i.e.\ from some $k$ onward) the smallest (resp. largest) index among the subtiles of $T^e_k$. 

We are interested only in orders of type $\Z$. Hence, the following definition and theorem are of importance.

\begin{defn}
	Let $\TTT$ be an ordered tiling system. The elements $\bT\in\TTT$ which induce an ordering of $G$ of type $\Z$ will be called \emph{straight}. We will denote
	\begin{equation*}
		\TTT_\mathsf{STR}=\{\bT\in\TTT:\bT\text{ is straight}\}.
	\end{equation*}
\end{defn}

It is clear that $\bT\in\TTT_{\mathsf{STR}}$ implies that $\bT$ is in general position.

For $\bT\in\TTT_\mathsf{STR}$, the order intervals with respect to $\prec_{_{\bT}}$ will be denoted by $[a,b]^{\bT}$ (rather than $[a,b]^{\prec_{_{\bT}}}$).
The next theorem establishes the uniform F\o lner property of the family $\{\prec_{_{\bT}}:\bT\in\TTT_\mathsf{STR}\}$.

\begin{thm}\label{oi}
	Let $\TTT$ be an ordered tiling system. For any finite subset $K\subset G$ and any $\eps>0$ there exists $n$ such that for any 
	$\bT\in\TTT_\mathsf{STR}$, any order-interval of length at least $n$ with respect to the order $\prec_{_{\bT}}$ on $G$, is $(K,\eps)$-\inv. 
\end{thm}

\begin{proof}
	Let $k$ be such that any shape $S\in\S_k$ is $(K,\frac\eps2)$-\inv\ and let $N$ denote the largest cardinality of an $S\in\S_k$. Put $n=\lceil\frac{2N}\delta\rceil$, where $\delta$ will be specified in a~moment. Fix any $\bT=(\T_k)_{k\in\N}\in\TTT_\mathsf{STR}$. By the definition of $\prec_{_{\bT}}$, any order-interval $[a,b]^{\bT}$ of length $n$ or larger is a union of complete tiles of $\T_k$ and perhaps a prefix and suffix which are subsets of some tiles of $\T_k$. The joint cardinality of the suffix and prefix is less than $2N$. Since the property of being $(K,\frac\eps2)$-\inv\ is preserved by finite disjoint union of sets, the union of complete tiles of $\T_k$ contained in $[a,b]^{\bT}$ is $(K,\frac\eps2)$-\inv. The prefix and suffix constitute at most a fraction~$\delta$ of the whole interval. It is now clear that for a sufficiently small $\delta$ (in fact any $0<\delta\le\frac\eps{2(|K|+1)}$ will do), the entire interval $[a,b]^{\bT}$ is $(K,\eps)$-\inv. 
\end{proof}

\begin{lem}\label{large}
	Let $\TTT$ be an ordered tiling system. Then the set $\TTT_\mathsf{STR}$
	is residual and has measure $1$ for every invariant Borel probability measure on $\TTT$.
\end{lem}

\begin{proof}
	The set of $\bT=(\T_k)_{k\in\N}\in\TTT$ which are in general position equals
	\begin{equation*}
		\bigcap_{g\in G}\bigcup_{k\in\N}\{\bT\in\TTT:g\in T_k^e\}.
	\end{equation*}
	Further, observe that an element $\bT$ is not straight if and only if either it is not in general position or, from some $k$ onward, the central tile $T^e_{k-1}$ of $\T_{k-1}$ equals the first subtile of $T^e_k$, or, from some $k$ onward, $T^e_{k-1}$ equals the last subtile of $T^e_k$. Formally, the set of $\bT$ which are not straight equals
	\begin{multline}\label{ns}
		\bigcup_{g\in G}\bigcap_{k\in\N}\{\bT\in\TTT:g\notin T_k^e\}
		\cup\bigcup_{k_0\in\N}\ \bigcap_{k\ge k_0}\ \bigcup_{S\in\S_{k-1}}\ \bigcup_{c\in S^{-1}}\{\bT\in\TTT: T_{k-1}^e=S'_1c^S_1c\}\\ \cup \bigcup_{k_0\in\N}\ \bigcap_{k\ge k_0}\ \bigcup_{S\in\S_{k-1}}\ \bigcup_{c\in S^{-1}}\{\bT\in\TTT: T^e_{k-1}=S'_{l(S)}c^S_{l(S)}c\},
	\end{multline}
	where $T^e_k = Sc=\bigsqcup_{i=1}^{l(S)}S'_ic_i^Sc$ is the decomposition of the central tile of $\T_k$ into subtiles (in the notation of \eqref{dect}). In \eqref{ns}, the three sets in curly brackets are clearly clopen, and so are the finite unions over $S$ and $c$. Thus the set defined by the entire formula is of type~$F_\sigma$. We have shown that the set of straight elements $\bT$ is of type $G_\delta$. That it is dense, and thus residual, will follow once we prove its largeness in the sense of measure (in a minimal system every set of full measure for at least one \im\ is dense).
	
	We pass to proving largeness in the sense of measure. Let $(K_k)_{k\in\N}$ be an increasing (by inclusion) \sq\ of finite subsets of $G$ such that $\bigcup_{k\in\N}K_k =G$. An~element $\bT=(\T_k)_{k\in\N}$ which satisfies, for infinitely many indices $k$, the condition, that $e$ is contained in the $K_k$-core of the central tile $T_k^e$ of $\T_k$, is in general position. Indeed, in such case $T_k^e$ contains $K_k$ for infinitely many $k$, hence the union of all central tiles equals $G$.\footnote{Observe that since $T^e_k$ contains the unit we have $e\in (T_k^e)_{K_k}\iff K_k\subset T^e_k$.} 
	
	Clearly, speeding up the tiling system does not affect the set of straight elements, hence, in this proof, we are free to speed up the tiling system as much as we need (from now on, $\TTT$ will denote the tiling system after speeding up). By speeding up we can achieve that every $\bT=(\T_k)_{k\in\N}$ satisfies the following condition: for every $k\in\N$ and every $S\in\S_k$, the union of the following subtiles of $S$: 
	\begin{itemize}
		\item the first and last one (i.e.\ $S'_1c^S_1$ and $S'_{l(S)}c^S_{l(S)}$ in the 
		notation of \eqref{decs}), and 
		\item those not contained in the $K_k$-core of $S$,
	\end{itemize}
	has cardinality at most $\frac{|S|}{2^k}$. If so, then for any $\bT\in\TTT$ and for every $k\in\N$, the union of all subtiles of the tiles $T$ of $\T_k$ which are either the first or last in the enumeration of the subtiles of $T$, or are not contained in the $K_k$-core of $T$, has upper Banach density\footnote{The upper Banach density of a set $D\subset G$ equals $\inf_{F}\sup_{g\in G}\frac{|D\cap Fg|}{|F|}$, where $F$ ranges over finite subsets of $G$. For more information about upper Banach density and its relation with the invariant measures see e.g.\ \cite[Section~6.2]{DZ}.} at most $2^{-k}$ (see e.g.\ \cite[Lemma~4.15]{DZ}). This, in turn, implies that for any invariant Borel probability measure $\nu$ on $\TTT$, the measure of those $\bT$ for which~$e$ belongs to the first or last subtile of the $T_k^e$ or to a subtile not contained in the $K_k$-core of $T_k^e$, is less than $2^{-k}$ (see e.g.\  \cite[Proposition~6.10~(1)]{DZ}). By the Borel--Cantelli lemma, the set of $\bT$ for which the above event happens infinitely many times, has measure zero for $\nu$. By the formula \eqref{ns}, this set (just shown to be of universal \im\ zero) contains all $\bT\in\TTT$ which are not straight. 
\end{proof}

\begin{defn}\label{ufp}
	A multiorder $(\tilde\O,\nu,G)$ is \emph{uniformly F\o lner} if for any finite set $K\subset G$ and $\eps>0$ there exists $n\in\N$ such that for $\nu$-almost every order $\prec\ \in\tilde\O$, any order interval $[a,b]^\prec$ of length at least $n$ is $(K,\eps)$-\inv.
\end{defn}

\begin{thm}\label{muo} Let $\mu$ be an invariant probability measure on an ordered tiling system $\TTT$. The assignment $\bT\mapsto\ \prec_{_{\bT}}$ is a factor map from $(\TTT,\mu,G)$ to the multiorder $(\tilde{\O},\nu,G)$, where $\tilde{\O}=\{\prec_{_{\bT}}:\bT\in \TTT_{\mathsf{STR}}\}$ and $\nu$ is the image of $\mu$ by the above map. Moreover, $\tilde{\O}$ is a uniformly F\o lner multiorder on~$G$.
\end{thm}

We can now strengthen our Corollary \ref{34} (and Theorem \ref{exi}):
\begin{cor}\label{muo1}
	On any countable amenable group $G$ there exists a uniformly F\o lner multiorder of entropy zero.
\end{cor}
\begin{proof}
	It suffices to use a tiling system of entropy zero (whose existence is guaranteed by \cite[Theorem 5.2]{DHZ}), and create from it an ordered tiling system. This ordered tiling system also has entropy zero (because it is a \tl\ factor of the unordered tiling system), and so does the resulting multiorder (for the same reason). By Theorem \ref{muo}, this multiorder is uniformly F\o lner.
\end{proof}

\begin{proof}[Proof of Theorem \ref{muo}]
	The assignment $\bT\mapsto\ \prec_{_{\bT}}$ is Borel-measurable and satisfies the equivariance condition $g(\bT)\mapsto g(\prec_{_{\bT}})$. Indeed, observe that given $\bT\in\TTT_{\mathsf{STR}}$ and an element $g\in G$, the order $\prec_{_{g(\bT)}}$ is obtained from $\prec_{_{\bT}}$ by shifting:
	\begin{equation*}
		a \prec_{_{g(\bT)}} b \iff ag \prec_{_{\bT}} bg,
	\end{equation*}
	which means that $\prec_{_{g(\bT)}}=g(\prec_{_{\bT}})$ as in Definition \ref{mo}.
	By Lemma \ref{large}, the set $\TTT_{\mathsf{STR}}$ is $G$-invariant, and carries all \im s of the system $\TTT$. Thus the family of orders $\{\prec_{_{\bT}}:\bT\in \TTT_{\mathsf{STR}}\}$ is also $G$-\inv\ and it carries precisely the \im s which are images (via the factor map $\bT\mapsto\ \prec_{_{\bT}}$) of the \im s supported by $\TTT$. Since $\TTT$ is a compact metric space on which $G$ acts by homeomorphisms, the collection of \im s on $\TTT$ is nonempty and thus so is the collection of \im s supported by the family $\{\prec_{_{\bT}}:\bT\in \TTT_{\mathsf{STR}}\}$. The uniform F\o lner property was proved in Theorem \ref{oi}.
\end{proof}

\medskip

We illustrate the idea of a tiling-based multiorder with three examples.

\begin{exam}\label{56}
	Let $G=(\Z,+)$. Let $\TTT$ consist of $\bT=(\T_k)_{k\in\N}$ such that each $\T_k$ partitions $\Z$ into intervals of equal lengths, say $2^k$ (by congruency, this condition determines $\TTT$ completely, the tiling system is conjugate to the dyadic odometer). 
	Let the subtiles of each tile $T$ be enumerated from left to right. Then, for every $\bT\in\TTT_{\mathsf{STR}}$, the order $\prec_{_{\bT}}$ coincides with the natural order in $\Z$. 
\end{exam}
Let us mention, that there are (countably many) non-straight elements $\bT\in\TTT$. They fail the ``in general position'' condition and they generate partial orders which coincide with the natural order on each of the halflines $(-\infty, n]$ and $[n+1,\infty)$ (for some $n\in\Z$), while elements from different halflines are incomparable.
\begin{exam}\label{Z-ord}
	Let $G$ and $\TTT$ be as in Example \ref{56} and let us order the two subtiles of the shape of order $k$ from left to right for even $k$ and from right to left for odd $k$. Then we will get a multiorder constisting of non-standard orders as in the figure below.
	\begin{figure}[h]
		\includegraphics[width=0.9\textwidth]{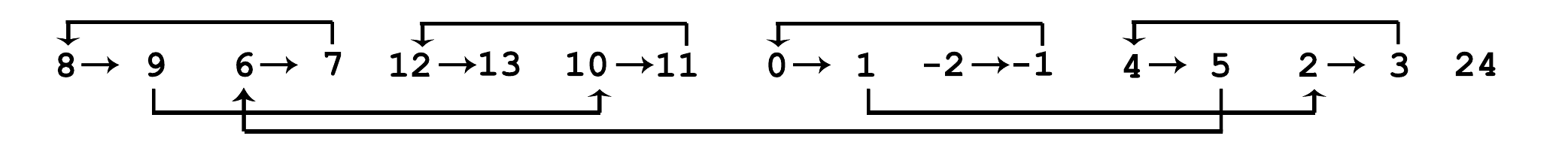}
		\caption{An example of a non-standard order on $\Z$}
	\end{figure}
\end{exam}
\begin{exam}
	Let $G=(\Z^2,+)$. Let $\TTT$ be such that, for each $k\ge 0$, the family of shapes $\S_{k+1}$ consists of four squares of equal dimensions $2^{k+1}\times 2^{k+1}$, identical as sets, but with different labels, say, $\sqcup_{k+1}$, $\sqsubset_{k+1}$, $\sqsupset_{k+1}$ and $\sqcap_{k+1}$. Each shape is subdivided into four subtiles (four identical squares), representing three of the available shapes $\sqcup_k$, $\sqsubset_k$, $\sqsupset_k$ and $\sqcap_k$ (one of them appearing twice). The shapes are subdivided as follows
	(the numeric matrices show the enumeration of the subtiles):
	\begin{gather*}
		\sqcup_{k+1}=
		\begin{bmatrix}
			\sqsupset_k &\sqsubset_k\\
			\sqcup_k & \sqcup_k
		\end{bmatrix} 
		\begin{bmatrix}
			1 & 4\\
			2 & 3
		\end{bmatrix},\ 
		\sqsubset_{k+1}\,=
		\begin{bmatrix}
			\sqsubset_k &\sqcup_k\\
			\sqsubset_k &\sqcap_k
		\end{bmatrix}
		\begin{bmatrix}
			3 & 4\\
			2 & 1
		\end{bmatrix},
		\\ 
		\sqsupset_{k+1}\,=
		\begin{bmatrix}
			\sqcup_k &\sqsupset_k\\
			\sqcap_k & \sqsupset_k
		\end{bmatrix}
		\begin{bmatrix}
			1 & 2\\
			4 & 3
		\end{bmatrix},\
		\sqcap_{k+1}=
		\begin{bmatrix}
			\sqcap_k &\sqcap_k\\
			\sqsupset_k & \sqsubset_k
		\end{bmatrix}
		\begin{bmatrix}
			3 & 2\\
			4 & 1
		\end{bmatrix}.
	\end{gather*}
	Under this rule, the ordering of $\Z^2$ follows the familiar pattern of the so-called \emph{Hilbert space-filling curve} (see Figure~\ref{hilbert_fig}). However, instead of making the curve denser and denser in each step (as it is done in the Hilbert curve's construction), we extend it to larger and larger squares in $\Z^2$, eventually filling $\Z^2$ entirely (as long as $\bT$ is straight). It can be checked that each $\bT\in\TTT_{\mathsf{STR}}$ generates a different bi-infinite Hilbert curve.
\end{exam}

\begin{figure}[h]
	\includegraphics[width=5cm]{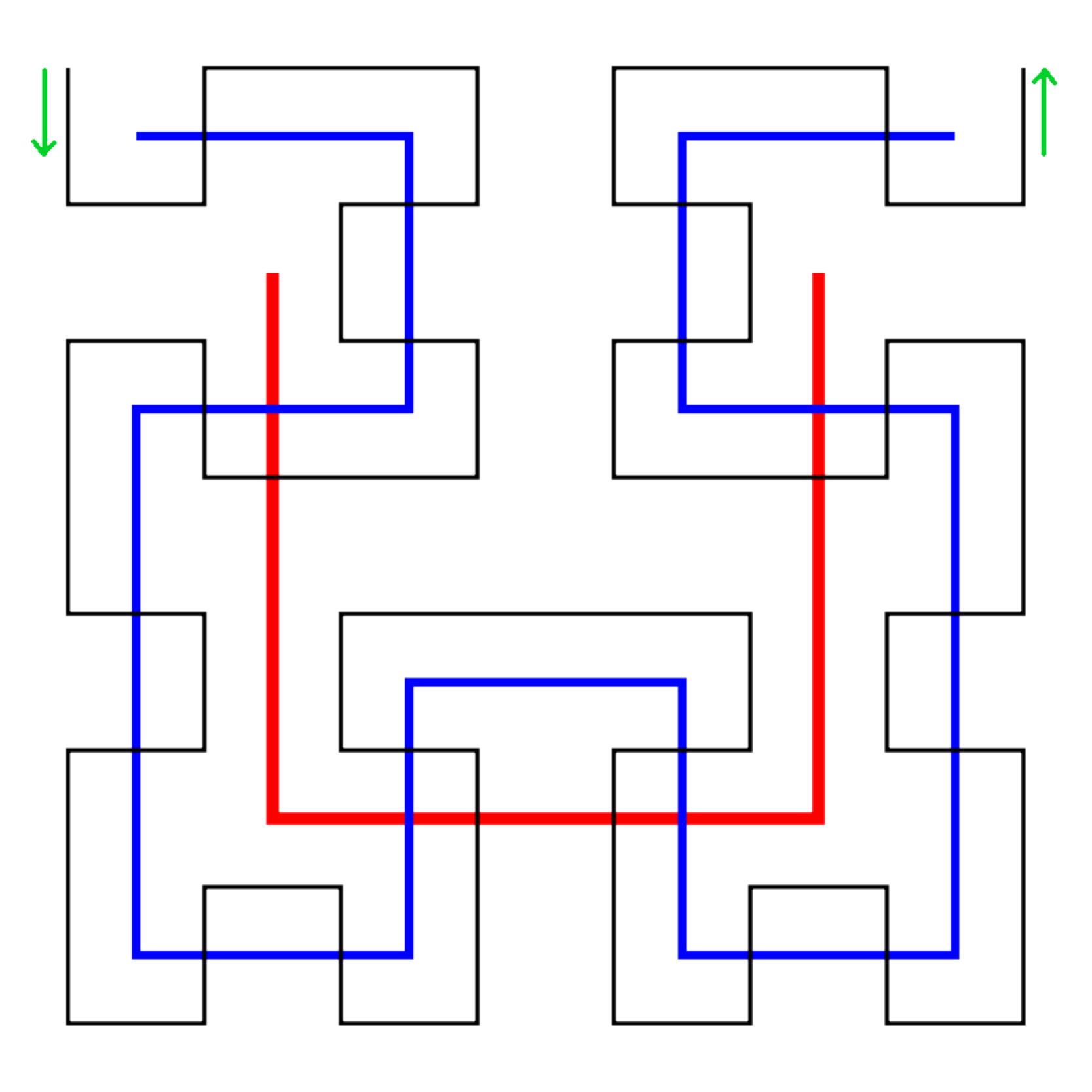}
	\caption{The Hilbert curve pattern}
	\label{hilbert_fig}
\end{figure}
\end{appendix}

\section*{Acknowledgements}
The authors would like to thank Tom Meyerovitch for an inspiring discussion in which he predicted Theorem \ref{TM}. He also provided an outline of the proof, and suggested the connection between multiorders and orbit equivalence to $\Z$-actions. We also thank Alexandre Danilenko for his comments in which he indicated alternative proofs of several results presented in this paper, and for a finding a flaw in the first version of Theorem~\ref{R-s}. Finally, we thank the anonymous referee for valuable suggestions which helped to improve the paper.

Part of the work was carried out during a visit of Piotr Oprocha and Guohua
Zhang in 2019 to the Faculty of Pure and Applied Mathematics, Wroc\l aw University of Science and Technology.
Tomasz Downarowicz is supported by National Science Center, Poland (Grant HARMONIA
No. 2018/30/M/ST1/00061) and by the Wroc\l aw University of Science and Technology.
Piotr Oprocha was supported by National Science Centre, Poland (Grant No. 2019/35/B/ST1/02239). Guohua Zhang is supported by National Natural
Science Foundation of China (Grants No. 11722103 and 11731003).

\end{document}